\documentclass[a4paper,reqno,12pt]{amsart}

\usepackage[all, 2cell, knot,]{xy} \UseAllTwocells \SilentMatrices
\usepackage{xspace}
\usepackage{amsmath}
\usepackage{amstext}
\usepackage{amsfonts}
\usepackage[mathscr]{euscript}
\usepackage{amscd}
\usepackage{latexsym}
\usepackage{amssymb}
\usepackage{layout}
\usepackage{amsthm}
\usepackage{amsfonts}
\usepackage{amsxtra}
\usepackage{color}
\usepackage[usenames,dvipsnames,svgnames,table]{xcolor}
\usepackage{graphics}
\usepackage{bm}
\usepackage{tabulary}
\usepackage{fancyhdr}
\usepackage[bookmarks=true,hyperfootnotes=false]{hyperref}
\hypersetup{
			colorlinks=true,
			linkcolor=blue,
			anchorcolor=black,
			citecolor=green,
			urlcolor=black,
}
\theoremstyle{plain}
\newtheorem{theorem}{Theorem}[section]
\newtheorem{lemma}[theorem]{Lemma}
\newtheorem{proposition}[theorem]{Proposition}
\newtheorem{corollary}[theorem]{Corollary}
\theoremstyle{definition}
\newtheorem{definition}[theorem]{Definition}
\newtheorem{example}[theorem]{Example}

\newtheorem{notation}[theorem]{Notation}
\newtheorem{remark}[theorem]{Remark}

\newcommand{\clA}{\mathcal{A}}

\newcommand{\clC}{\mathcal{C}}
\newcommand{\clD}{\mathcal{D}}
\newcommand{\clE}{\mathcal{E}}

\newcommand{\clI}{\mathcal{I}}

\newcommand{\clU}{\mathcal{U}}

\newcommand{\za}{\alpha}
\newcommand{\zb}{\beta}
\newcommand{\zg}{\gamma}

\newcommand{\zD}{\Delta}

\newcommand{\zve}{\varepsilon}

\newcommand{\pt}{\partial}

\newcommand{\Id}{\operatorname{Id}}

\newcommand{\pr}{\operatorname{pr}}
\newcommand{\nequ}{\mbox{$n$-equivalence}}
\newcommand{\equ}[1]{\mbox{$(#1)$-equivalence}}
\newcommand{\bsim}{/\!\!\sim}

\newcommand{\nid}{\noindent}
\newcommand{\bk}{\bigskip}
\newcommand{\mk}{\medskip}
\newcommand{\sk}{\smallskip}
\newcommand{\ovl}[1]{\overline{#1}}
\newcommand{\ovll}[1]{\overset{=}{#1}}
\newcommand{\up}[1]{^{(#1)}}
\newcommand{\lo}[1]{_{(#1)}}
\newcommand{\rw}{\rightarrow}
\newcommand{\Rw}{\Rightarrow}
\newcommand{\lw}{\leftarrow}

\newcommand{\xrw}{\xrightarrow} 
\newcommand{\xlw}{\xleftarrow} 
\newcommand{\hxrw}[1]{\xymatrix{\ \ar@{^{(}->}^{#1}[r] & \ }}
\newcommand{\tiund}[1]{{\times}_{#1}\:}
\newcommand{\pro}[3]{#1\tiund{#2}\overset{#3}{\cdots}\tiund{#2}#1}
\newcommand{\tens}[2]{#1\,\tiund{#2}\,#1}
\newcommand{\uset}[2]{\underset{#1}{#2}}
\newcommand{\oset}[2]{\overset{#1}{#2}}

\newcommand{\mi}{\text{-}}
\newcommand{\nm}{(n-1)}

\newcommand{\cop}{\textstyle{\,\coprod\,}}
\newcommand{\seq}[3]{{#1}_{#2}...{#1}_{#3}}
\newcommand{\seqc}[3]{{#1}_{#2},...,{#1}_{#3}}

\newcommand{\dop}[1]{\Delta^{{#1}^{op}}}
\newcommand{\Dop}{\Delta^{op}}
\newcommand{\Dnop}{\Delta^{{n}^{op}}}
\newcommand{\Dmenop}{\Delta^{{n-1}^{op}}}
\newcommand{\cat}[1]{\mbox{$\mathsf{Cat^{#1}}$}}
\newcommand{\Cat}{\mbox{$\mathsf{Cat}\,$}}

\newcommand{\Gpd}{\mbox{$\mathsf{Gpd}$}}

\newcommand{\cathd}[1]{\mbox{$\mathsf{Cat_{hd}^{#1}}$}}

\newcommand{\catwg}[1]{\mbox{$\mathsf{Cat_{wg}^{#1}}$}}
\newcommand{\tawg}[1]{\mbox{$\mathsf{Ta_{wg}^{#1}}$}}
\newcommand{\lta}[1]{\mbox{$\mathsf{LTa_{wg}^{#1}}$}}

\newcommand{\lnta}[2]{\mbox{$\mathsf{L\lo{#1}Ta_{wg}^{#2}}$}}

\newcommand{\segpsc}[2]{\mbox{$\mathsf{SegPs}$}\funcat{#1}{#2}}
\newcommand{\Ps}{\mbox{$\mathsf{Ps}$}}
\newcommand{\psc}[2]{\mbox{$\mathsf{Ps}$}\funcat{#1}{#2}}
\newcommand{\PsTalg}{\mbox{\sf Ps-}T\mbox{\sf -alg}}
\newcommand{\muk}{\mu_k}
\newcommand{\hmu}[1]{\hat\mu_{#1}}
\newcommand{\hmuk}{\hat{\mu}_k}

\newcommand{\Nn}{N_{(n)}}
\newcommand{\N}[1]{N_{(#1)}}
\newcommand{\Nb}[1]{N_{(#1)}}
\newcommand{\Nu}[1]{N^{(#1)}}
\newcommand{\di}[1]{d^{(#1)}}
\newcommand{\dn}{d^{(n)}}

\newcommand{\p}[1]{p^{(#1)}}

\newcommand{\op}[1]{\bar{p}^{(#1)}}
\newcommand{\q}[1]{q^{(#1)}}

\newcommand{\qn}{q^{(n)}}

\newcommand{\zgu}[1]{\zg^{(#1)}}
\newcommand{\Tan}{\mbox{$\mathsf{Ta^{n}}$}} 
\newcommand{\ta}[1]{\mbox{$\mathsf{Ta^{#1}}$}} 
\newcommand{\Set}{\mbox{$\mathsf{Set}$}}
\newcommand{\St}{St\,}

\newcommand{\Tr}{Tr\,}

\newcommand{\uk}{\underline{k}}
\newcommand{\ur}{\underline{r}}
\newcommand{\us}{\underline{s}}

\newcommand{\funcat}[2]{[\Delta^{{#1}^{op}},#2]}

\newcommand{\Lb}[1]{\mbox{$\mathsf{L_{(#1)}}$}}
\newcommand{\nfol}{$n$-fold }

\begin{document}

\title [\tiny{Weakly globular Tamsamani $n$-categories and their...}]{Weakly globular Tamsamani $\bm{N}$-categories and their rigidification }

\author[\tiny{Simona Paoli}]{Simona Paoli}
 \address{Department of Mathematics, University of Leicester,
LE17RH, UK}
 \email{sp424@le.ac.uk}

\date{3 September 2016}

\keywords{$n$-fold category, pseudo-functors, weak $n$-category, multi-simplicial sets}

\subjclass[2010]{18Dxx}


\begin{abstract}
We introduce a new class of higher categorical structures called weakly globular Tamsamani $n$-categories. These generalize the Tamsamani-Simpson model of higher categories by using the new paradigm of weak globularity to weaken higher categorical structures. We prove this new structure is suitably equivalent to a simpler one previously introduced by the author, called  weakly globular $n$-fold categories.
\end{abstract}

\maketitle


\section{Introduction}\label{sec-intro}

In this paper we introduce a new higher categorical structure, called a weakly globular Tamsamani $n$-category and we prove its equivalence with a simpler higher categorical structure introduced by the author in \cite{Pa2}, called a weakly globular $n$-fold category.

Higher category theory witnessed an explosion of interest in recent years and found applications to disparate fields, such as homotopy theory \cite{Be2} \cite{L2}, algebraic geometry \cite{Simp}, mathematical physics \cite{Lu3}, logic and computer science \cite{HTT},\cite{Voe}.

In a category there are objects, arrows, identity arrows, and composition of arrows that are associative and unital. Higher categories generalize categories by admitting higher arrows (also called higher cells) and compositions between them. The behaviour of these compositions determine two main classes: strict and weak higher categories.

For each of these classes, there are higher categories which admits cells in every dimensions, and those that have cells in dimensions only up $0$ to $n$. The former have been studied extensively in relation to applications to homotopy theory, giving rise to several notions of $(\infty,n)$-category \cite{BeRe}, \cite{Be3}, \cite{Be2}, \cite{Bk1}, \cite{Lu3}, \cite{L2} as well as weak $\omega$-category \cite{Ve}.

In this work we concentrate on the second 'truncated' case, with cells in dimensions $0$ up to $n$. This is close to one of the original motivations for the development of higher categories, namely the modelling of the building blocks of spaces, the $n$-types, and it is of fundamental importance for higher category theory. It also leads to applications to homotopy theory in the search for a combinatorial description of the $k$-invariants of spaces and of simplicial categories.

In a strict higher category, compositions are associative and unital, and there is a simple way to describe strict $n$-categories via iterated enrichment. Although simple to define, strict $n$-categories are insufficient for many applications and the wider class of weak $n$-categories is needed.

In a weak $n$-category, higher cells compose in a way that is associative and unital only up to an invertible cell in the next dimension, and these associativity and unit isomorphisms are suitably compatible or coherent.

In low dimensions, the prototype of a weak $n$-category is the classical the notion of bicategory \cite{Ben} (when $n$=2) and tricategory \cite{GPS}(when $n$=3).
In dimension $n>3$, writing explicitly the coherence data for a weak $n$-category is too complex. Instead, different combinatorial models of weak $n$-category were developed, using a variety of techniques such as operads \cite{B}, \cite{Cheng2}, \cite{Lu2}, simplicial sets \cite{Simp}, \cite{Ta}, opetopes  \cite{Cheng1} and many others.

 A good notion of weak $n$-category needs to satisfy the homotopy hypothesis of giving an algebraic model of the building blocks of spaces (the $n$-types) when the cells in the structure are weakly invertible, that is in the weak $n$-groupoid case.

A very important model of higher categories was developed by Tamsamani and Simpson \cite{Simp}, \cite{Ta}. This model is based on the combinatorics of multi-simplicial sets. The latter are  combinatorial models of spaces and feature prominently in homotopy theory \cite{Jard}.

The idea of the Tamsamani-Simpson model stems from a well-known connection between categories and simplicial sets, via the nerve functor
 \begin{equation*}
    N:\Cat \rw \funcat{}{\Set}\;.
\end{equation*}
This functor is fully faithful and a characterization of its image can be given in terms of Segal maps. Given a simplicial set $X$ the Segal maps for each $k\geq 2$ are maps of sets
\begin{equation*}
  X_k \rw \pro{X_1}{X_0}{k}.
\end{equation*}

Then $X$ is the nerve of a category if and only if all the Segal maps are isomorphisms. Thus the Segal maps are natural candidates for the composition in a category, and the condition that they are isomorphisms determines the structure of a category.

Segal maps can be defined for any $n$-fold simplicial set $X\in \funcat{n}{\Set}$. Namely, viewing $X$ as a simplicial object in $(n-1)$-fold simplicial sets in $n$ different directions affords corresponding Segal maps. These are candidates for the compositions of the higher cells in the structures.

In the Tamsamani-Simpson model $\ta{n}$, the higher cells are encoded as discrete structures (that is, sets) in the respective dimensions. This is called the \emph{globularity condition}, since it gives rise to the globular shape of the higher cells in the structure.

If in addition to the globularity condition we require the Segal maps in all simplicial directions to be isomorphisms we obtain a strict $n$-category, described in multi-simplicial language.

The Tamsamani-Simpson model is a suitable weakening of this construction. Namely, the globularity condition is maintained, but the Segal maps are no longer isomorphisms but suitably defined higher categorical equivalences.

In this paper we introduce a new Segal-type model of higher categories based on multi-simplicial structures, called a weakly globular Tamsamani $n$-category and denoted by $\tawg{n}$. This model uses a new paradigm to weaken higher categorical structures, which is the notion of \emph{weak globularity}. This notion was introduced by the author in \cite{Pa2} with the category $\catwg{n}$ of weakly globular \nfol categories.

In the weakly globular approach to higher categories the globularity condition is replaced by weak globularity: the cells in each dimension ($0$ up to $n$) no longer form a set but have a categorical structure on their own which is equivalent in a suitable sense to a discrete one. More precisely, it is a homotopically discrete $n$-fold category in the sense defined by the author in
\cite{Pa1}.

The behaviour of the compositions is controlled by the induced Segal maps, which for $X\in\tawg{n}$ have the form for each $k\geq 2$\;
\begin{equation*}
  X_k\rw\pro{X_1}{X_0^d}{k}.
\end{equation*}

Here $X_0^d$ is the discretization of the homotopically discrete object $X_0$. In a weakly globular Tamsamani $n$-category the  induced Segal maps are required to be suitable higher categorical equivalences. The formal definition of $X$ is given by induction on dimension. The Segal maps are in general not isomorphisms.

Weakly globular $n$-fold categories satisfy the weak globularity condition and their induced Segal maps are higher categorical equivalences but, being $n$-fold categories, all their Segal maps are isomorphisms.

There are embeddings
\begin{equation*}
  \catwg{n}\subset\tawg{n} \qquad \text{and}\qquad \ta{n}\subset\tawg{n}\;.
\end{equation*}
Thus weakly globular Tamsamani $n$-categories are a generalization both of the Tamsamani-Simpson model and of weakly globular $n$-fold categories.

The main result of this paper, Theorem \ref{the-funct-Qn}, asserts the existence of a functor, called \emph{rigidification}
\begin{equation*}
  Q_n:\tawg{n}\rw\catwg{n}
\end{equation*}
such that for each $X\in\tawg{n}$ there is a suitable higher categorical equivalence $Q_n X\rw X$. This result means that $X$ can be approximated with the more rigid and therefore simpler structure $Q_n X$. In particular, this implies (see Corollary \ref{cor-the-funct-Qn}) that the two categories $\catwg{n}$ and $\tawg{n}$ are equivalent after localization with respect to the $n$-equivalences.

Rigidification of structures defined up to homotopy is an important topic in homotopy theory, where it is often tackled with model theoretic techniques \cite{Badz}. In the higher categorical setting, rigidification results have been proved  for $(\infty,1)$-categories \cite{Be2} but not for the more complex case of weak $n$-categories for general $n>0$. Our result is an important step in the study of rigidification in the general higher categorical setting, and paves the way for further developments and applications.

In a subsequent paper \cite{Pa4} we will show that the rigidification functor $Q_n$ leads to a suitable equivalence between $\catwg{n}$ and the Tamsamani $n$-categories $\ta{n}$, exhibiting $\catwg{n}$ as a new model of weak $n$-categories satisfying, in particular, the homotopy hypothesis. This model will lead to new connections between different combinatorics modelling higher structures, as well as to applications to homotopy theory.

The rigidification functor factors through the category
 \begin{equation*}
   \segpsc{n-1}{\Cat}
 \end{equation*}
 of Segalic pseudo-functors introduced by the author in \cite{Pa2} whose strictification gives weakly globular \nfol categories (see \cite{Pa2}). That is, $Q_n$ is the composite
\begin{equation*}
    Q_n:\tawg{n}\xrw{\ \;\;} \segpsc{n-1}{\Cat}\xrw{\St} \catwg{n}\;.
\end{equation*}
where $\St$ is the classical strictification of pseudo-functors, see \cite{Lack}, \cite{PW}.
The functor from $\tawg{n}$ to Segalic pseudo-functors further factorizes as
\begin{equation*}
    \tawg{n}\xrw{P_n} \lta{n} \xrw{\Tr_{n}} \segpsc{n-1}{\Cat}
\end{equation*}
where $\lta{n}$ is a full subcategory of $\tawg{n}$ from which it is easy to construct a pseudo-functor using a general categorical technique known as of 'transport of structure along an adjunction'. The functor $P_n$ produces a functorial approximation (up to $n$-equivalence) of an object of $\tawg{n}$ with an object of $\lta{n}$.

This paper is organized as follows. In Section \ref{sec-prelim} we recall some basic multi-simplicial techniques used throughout the paper, the classical theory of strictification of pseudo-functors as well as the technique of transport of structure to produce pseudo-functors.

 In Section \ref{sec-WG-nfol-cat} we recall the definition of the category $\catwg{n}$ of weakly globular \nfol categories and of Segalic pseudo-functors introduced in \cite{Pa2}.

In Section \ref{sec-WG-Tam-cat} we introduce the category $\tawg{n}$ of weakly globular Tamsamani $n$-categories and of $n$-equivalences and we prove its main properties. In particular we show in Proposition \ref{pro-crit-lev-nequiv} a criterion for an $n$-equivalence in $\tawg{n}$ to be a levelwise equivalence of categories. This is used crucially in Section \ref{sec-cat-lta} in the proofs of Lemma \ref{lem-jn-alpha} and Theorem \ref{the-repl-obj-1}.

In Section \ref{sec-func-qn} we continue the study of the category $\tawg{n}$ and we focus on the important functor
\begin{equation*}
    \q{n}:\tawg{n}\rw \tawg{n-1}\;.
\end{equation*}
This functor generalizes to higher dimensions the connected components functor $q:\Cat\rw\Set$. For each $X\in\tawg{n}$ there is a map $\zg\up{n}: X\rw \di{n}\q{n}X$ where $\di{n}$ is the inclusion of $\tawg{n-1}$ in $\tawg{n}$ as a discrete structure in the top dimension. We study the properties of pullbacks along this map playing a crucial role in Section \ref{sec-cat-lta}.

In Section \ref{sec-cat-lta} we introduce the subcategory $\lta{n}\subset \tawg{n}$ and we show in Theorem \ref{the-repl-obj-1} and Corollary \ref{cor-repl-obj-2} that every object of $\tawg{n}$ is $n$-equivalent to an object of $\lta{n}$.

In Section \ref{sec-wg-tam-to-psefun}, Theorem \ref{the-XXXX}  we prove the existence of a functor
\begin{equation*}
    Tr_n: \lta{n}\rw \segpsc{n-1}{\Cat}
\end{equation*}

We then use the functor $Tr_{n}$ and the results of Section \ref{sec-cat-lta} to build the rigidification functor $Q_n$.

\bk

\textbf{Acknowledgements}: This work is supported by a Marie Curie International Reintegration Grant No 256341. I thank the Centre for Australian Category Theory for their hospitality and financial support during August-December 2015, as well as the University of Leicester for its support during my study leave.  I also thank the University of Chicago for their hospitality and financial support during April 2016.


\section{Preliminaries}\label{sec-prelim}
In this section we review some basic simplicial techniques that we will use throughout the paper as well as some categorical background on pseudo-functors and their strictification, and on a technique to produce pseudo-functors. The material in this section is well-known, see for instance \cite{Borc}, \cite{Jard}, \cite{lk}, \cite{PW}, \cite{Lack}.
\subsection{Simplicial objects}\label{sbs-simp-tech}
Let $\zD$ be the simplicial category and let $\Dnop$ denote the product of $n$ copies of $\Dop$. Given a category $\clC$, $\funcat{n}{\clC}$ is called the category of $n$-simplicial objects in $\clC$ (simplicial objects in $\clC$ when $n=1$).
\begin{notation}\label{not-simp}
    If $X\in \funcat{n}{\clC}$ and $\uk=([k_1],\ldots,[k_n])\in \Dnop$, we shall denote $X ([k_1],\ldots,[k_n])$ by $X(k_1,\ldots,k_n)$, as well as $X_{k_1,\ldots,k_n}$ and $X_{\uk}$. We shall also denote $\uk(1,i)=([k_1],\ldots,[k_{i-1}],1,[k_{i+1}],\ldots,[k_n]) \in \Dnop$ for $1\leq i\leq n$.

    Every $n$-simplicial object in $\clC$ can be regarded as a simplicial object in $\funcat{n-1}{\clC}$ in $n$ possible ways. For each $1\leq i\leq n$ there is an isomorphism
    \begin{equation*}
        \xi_i:\funcat{n}{\clC}\rw\funcat{}{\funcat{n-1}{\clC}}
    \end{equation*}
    given by
    \begin{equation*}
        (\xi_i X)_{r}(k_1,\ldots,k_{n-1})=X(k_1,\ldots,k_{i-1},r,k_{i+1},\ldots,k_{n-1})
    \end{equation*}
    for $X\in\funcat{n}\clC$ and $r\in\Dop$.
\end{notation}
\begin{definition}\label{def-fun-smacat}
    Let $F:\clC \rw \clD$ be a functor, $\clI$ a small category. Denote
    \begin{equation*}
        \ovl{F}:[\clI,\clC]\rw [\clI,\clD]
    \end{equation*}
    the functor given by
    \begin{equation*}
        (\ovl{F} X)_i = F(X(i))
    \end{equation*}
    for all $i\in\clI$.
\end{definition}
\begin{definition}\label{def-seg-map}
    Let ${X\in\funcat{}{\clC}}$ be a simplicial object in any category $\clC$ with pullbacks. For each ${1\leq j\leq k}$ and $k\geq 2$, let ${\nu_j:X_k\rw X_1}$ be induced by the map  $[1]\rw[k]$ in $\Delta$ sending $0$ to ${j-1}$ and $1$ to $j$. Then the following diagram commutes:
\begin{equation}\label{eq-seg-map}
\xymatrix@C=20pt{
&&&& X\sb{k} \ar[llld]_{\nu\sb{1}} \ar[ld]_{\nu\sb{2}} \ar[rrd]^{\nu\sb{k}} &&&& \\
& X\sb{1} \ar[ld]_{d\sb{1}} \ar[rd]^{d\sb{0}} &&
X\sb{1} \ar[ld]_{d\sb{1}} \ar[rd]^{d\sb{0}} && \dotsc &
X\sb{1} \ar[ld]_{d\sb{1}} \ar[rd]^{d\sb{0}} & \\
X\sb{0} && X\sb{0} && X\sb{0} &\dotsc X\sb{0} && X\sb{0}
}
\end{equation}

If  ${\pro{X_1}{X_0}{k}}$ denotes the limit of the lower part of the
diagram \eqref{eq-seg-map}, the \emph{$k$-th Segal map for $X$} is the unique map
$$
\muk:X\sb{k}~\rw~\pro{X\sb{1}}{X\sb{0}}{k}
$$
\noindent such that ${\pr_j\,\muk=\nu\sb{j}}$ where
${\pr\sb{j}}$ is the $j^{th}$ projection.
\end{definition}
\begin{definition}\label{def-ind-seg-map}

    Let ${X\in\funcat{}{\clC}}$ and suppose that there is a map $\zg: X_0 \rw X^d_0$ in $\clC$  $\zg: X_0 \rw X^d_0$  such that the limit of the diagram
\begin{equation*}
\xymatrix@R25pt@C16pt{
& X\sb{1} \ar[ld]_{\zg d\sb{1}} \ar[rd]^{\zg d\sb{0}} &&
X\sb{1} \ar[ld]_{\zg d\sb{1}} \ar[rd]^{\zg d\sb{0}} &\cdots& k &\cdots&
X\sb{1} \ar[ld]_{\zg d\sb{1}} \ar[rd]^{\zg d\sb{0}} & \\
X^d\sb{0} && X^d\sb{0} && X^d\sb{0}\cdots &&\cdots X^d\sb{0} && X^d\sb{0}
    }
\end{equation*}
exists; denote the latter by $\pro{X_1}{X_0^d}{k}$. Then the following diagram commutes, where $\nu_j$ is as in Definition \ref{def-seg-map}, and $k\geq 2$
\begin{equation*}
\xymatrix@C=20pt{
&&&& X\sb{k} \ar[llld]_{\nu\sb{1}} \ar[ld]_{\nu\sb{2}} \ar[rrd]^{\nu\sb{k}} &&&& \\
& X\sb{1} \ar[ld]_{\zg d\sb{1}} \ar[rd]^{\zg d\sb{0}} &&
X\sb{1} \ar[ld]_{\zg d\sb{1}} \ar[rd]^{\zg d\sb{0}} && \dotsc &
X\sb{1} \ar[ld]_{\zg d\sb{1}} \ar[rd]^{\zg d\sb{0}} & \\
X^d\sb{0} && X^d\sb{0} && X^d\sb{0} &\dotsc X^d\sb{0} && X^d\sb{0}
}
\end{equation*}
The \emph{$k$-th induced Segal map for $X$} is the unique map
\begin{equation*}
\hmuk:X\sb{k}~\rw~\pro{X\sb{1}}{X^d\sb{0}}{k}
\end{equation*}
such that ${\pr_j\,\hmuk=\nu\sb{j}}$ where ${\pr\sb{j}}$ is the $j^{th}$ projection.
\end{definition}
\subsection{$\mathbf{n}$-Fold internal categories}\label{sbs-nint-cat}

Let  $\clC$ be a category with finite limits. An internal category $X$ in $\clC$ is a diagram in $\clC$
\begin{equation}\label{eq-nint-cat}
\xymatrix{
\tens{X_1}{X_0} \ar^(0.65){m}[r] & X_1 \ar^{d_0}[r]<2.5ex> \ar^{d_1}[r] & X_0
\ar^{s}[l]<2ex>
}
\end{equation}
where $m,d_0,d_1,s$ satisfy the usual axioms of a category (see for instance \cite{Borc} for details). An internal functor is a morphism of diagrams like \eqref{eq-nint-cat} commuting in the obvious way. We denote by $\Cat \clC$ the category of internal categories and internal functors.

The category $\cat{n}(\clC)$ of \nfol categories in $\clC$ is defined inductively by iterating $n$ times the internal category construction. That is, $\cat{1}(\clC)=\Cat$ and, for $n>1$,
\begin{equation*}
  \cat{n}(\clC)= \Cat(\cat{n-1}(\clC)).
\end{equation*}

When $\clC=\Set$, $\cat{n}(\Set)$ is simply denoted by $\cat{n}$ and called the category of \nfol categories (double categories when $n=2$).

\subsection{Nerve functors}\label{sus-ner-funct}

There is a nerve functor
\begin{equation*}
    N:\Cat\clC \rw \funcat{}{\clC}
\end{equation*}
such that, for $X\in\Cat\clC$
\begin{equation*}
    (N X)_k=
    \left\{
      \begin{array}{ll}
        X_0, & \hbox{$k=0$;} \\
        X_1, & \hbox{$k=1$;} \\
        \pro{X_1}{X_0}{k}, & \hbox{$k>1$.}
      \end{array}
    \right.
\end{equation*}
When no ambiguity arises, we shall sometimes denote $(NX)_k$ by $X_k$ for all $k\geq 0$.

The following fact is well known:
\begin{proposition}\label{pro-ner-int-cat}
    A simplicial object in $\clC$ is the nerve of an internal category in $\clC$ if and only if all the Segal maps are isomorphisms.
\end{proposition}

By iterating the nerve construction, we obtain the multinerve functor
\begin{equation*}
    \Nn :\cat{n}(\clC)\rw \funcat{n}{\clC}\;.
\end{equation*}
\begin{definition}\label{def-discrete-nfold}
An internal $n$-fold category $X\in \cat{n}(\clC)$ is said to be discrete if $\Nn X$ is a constant functor.
\end{definition}

 Each object of $\cat{n}(\clC)$ can be considered as an internal category in $\cat{n-1}(\clC)$ in $n$ possible ways, corresponding to the $n$ simplicial directions of its multinerve. To prove this, we use the following lemma, which is a straightforward consequence of the definitions.

\begin{lemma}\label{lem-multin-iff}\
\begin{itemize}
      \item [a)] $X\in\funcat{n}{\clC}$ is the multinerve of an \nfol category in $\clC$ if and only if, for all $1\leq r\leq n$ and $[p_1],\ldots,[p_r]\in\Dop$, $p_r\geq 2$
\begin{equation}\label{eq-multin-iff}
\begin{split}
    &  X(p_1,...,p_r,\mi)\cong\\
    &\resizebox{1.0\hsize}{!}{$\cong\pro{X(p_1,...,p_{r-1},1,\mi)}{X(p_1,...,p_{r-1},0,\mi)}{p_r}$}
\end{split}
\end{equation}
      \item [b)] Let $X\in\cat{n}(\clC)$. For each $1\leq k\leq n$, $[i]\in\Dop$, there is $X_i\up{k}\in\cat{n-1}(\clC)$ with
\begin{equation*}
    \Nb{n-1}X_i\up{k} (p_1,\ldots,p_{n-1})=\Nn X(p_1,\ldots,p_{k-1},i,p_k,\ldots,p_{n-1})
\end{equation*}
    \end{itemize}
\end{lemma}
\begin{proof}\

  a) By induction on $n$. By Proposition \ref{pro-ner-int-cat}, it is true for $n=1$. Suppose it holds for $n-1$ and let $X\in\Cat(\cat{n-1}(\clC))$ with objects of objects (resp. arrows) $X_0$ (resp. $X_1$); denote $(NX)_p=X_p$. By definition of the multinerve
      \begin{equation*}
        (\Nb{n} X)(p_1,\ldots,p_r,\mi)=\Nb{n-1}X_{p_1}(p_2,\ldots,p_r,\mi)\;.
      \end{equation*}
      Hence using the induction hypothesis
\begin{align*}
       &\Nb{n}X(p_1...p_r\,\mi)=\Nb{n-1}X_{p_1}(p_2... p_r\,\mi)\cong\\
&\resizebox{1.0\hsize}{!}{$
\cong \pro{\Nb{n-1} X_{p_1}(p_2... p_{r-1}\,1\,\mi)}
         {\Nb{n-1} X_{p_1}(p_2... p_{r-1}\,0\,\mi)}{p_r}=$}\\
&\resizebox{1.0\hsize}{!}{
          $=\pro{\Nb{n} X(p_1... p_{r-1}\,1\,\mi)}
         {\Nb{n} X(p_1... p_{r-1}\,0\,\mi)}{p_r}.$}
\end{align*}
Conversely, suppose $X\in\funcat{n}{\clC}$ satisfies \eqref{eq-multin-iff}. Then for each $[p]\in\Dop$, $X(p,\mi)$ satisfies \eqref{eq-multin-iff}, hence
\begin{equation*}
    X(p,\mi)=\Nb{n-1}X_p
\end{equation*}
for $X_p\in\cat{n-1}(\clC)$. Also, by induction hypothesis
\begin{equation*}
    X(p,\mi)=\pro{X(1,\mi)}{X(0,\mi)}{p}\;.
\end{equation*}
Thus we have the object $X\in\cat{n}(\clC)$ with objects $X_0$, arrows $X_1$ and $X_p=X(p,\mi)$ as above.

\bigskip
b) By part a), there is an isomorphism for $p_r\geq 2$
\newpage
\begin{flalign*}
       &\Nb{n}X(p_1...p_n)=&
\end{flalign*}
\begin{equation*}
\resizebox{1.0\hsize}{!}{$\pro{\Nb{n}X(p_1...p_{r-1}\, 1 ...p_n)}{\Nb{n}X(p_1...p_{r-1}\, 0 ...p_n)}{p_r}$}\;.
\end{equation*}
In particular, evaluating this at $p_k=i$, this is saying the $(n-1)$-simplicial group taking $(p_1...p_n)$ to $\Nb{n}X(p_1...p_{k-1}\, i ...p_{n-1})$ satisfies condition \eqref{eq-multin-iff} in part a). Hence by part a) there exists $X_i\up{k}$ with
\begin{equation*}
    \Nb{n-1}X_i\up{k}(p_1...p_{n-1})=\Nb{n}X(p_1...p_{k-1}\, i ...p_{n-1})
\end{equation*}
as required.
\end{proof}
\begin{proposition}\label{pro-assoc-iso}
    For each $1\leq k\leq n$ there is an isomorphism $\xi_k:\cat{n}(\clC)\rw \cat{n}(\clC)$ which associates to $X=\cat{n}(\clC)$ an object $\xi_k X$ of $\Cat(\cat{n-1}(\clC))$ with
    \begin{equation*}
        (\xi_k X)_i=X_i\up{k}\qquad i=0,1
    \end{equation*}
    with $X_i\up{k}$ as in Lemma \ref{lem-multin-iff}.
\end{proposition}
\begin{proof}
Consider the object of $\funcat{}{\funcat{n-1}{\clC}}$ taking $i$ to the $\nm$-simplicial object associating to $(\seqc{p}{1}{n-1})$ the object
\begin{equation*}
    \Nb{n}X(p_1...p_{k-1}\, i \, p_{k+1}...p_{n-1})\;.
\end{equation*}
By Lemma \ref{lem-multin-iff} b), the latter is the multinerve of $X_i\up{k}\in \cat{n-1}(\clC)$. Further, by Lemma \ref{lem-multin-iff} a), we have
\begin{equation*}
    \Nb{n-1}X_i\up{k}\cong \pro{\Nb{n-1}X_1\up{k}}{\Nb{n-1}X_0\up{k}}{i}\;.
\end{equation*}
Hence $\Nb{n}X$ as a simplicial object in $\funcat{n-1}{\clC}$ along the $k^{th}$ direction, has
\begin{equation*}
    (\Nb{n}X)_i=
    \left\{
      \begin{array}{ll}
        \Nb{n-1}X_i\up{k}, & \hbox{$i=0,1$;} \\
        \Nb{n-1}(\pro{X_1\up{k}}{X_0\up{k}}{i}), & \hbox{for $i\geq 2$.}
      \end{array}
    \right.
\end{equation*}
This defines $\xi_k X\in\Cat(\cat{n-1}(\clC))$ with
\begin{equation*}
    (\xi_k X)_i=\Nb{n-1}X_i\up{k}\qquad i=0,1\;.
\end{equation*}
We now define the inverse for $\xi_k$. Let $X\in\Cat(\cat{n-1}(\clC))$, and let $X_i=\pro{X_1}{X_0}{i}$ for $i\geq 2$. The $n$-simplicial object $X_{k}$ taking $(p_1,\ldots,p_n)$ to
\begin{equation*}
    \Nb{n}X_{p_k}(p_1...p_{k-1}p_{k+1}...p_n)
\end{equation*}
satisfies condition \eqref{eq-multin-iff}, as easily seen. Hence by Lemma \ref{lem-multin-iff} there is $\xi'_k X\in\cat{n}\clC$ such that $\Nb{n }\xi'_k X=X_k$. It is immediate to check that $\xi_k$ and $\xi'_k$ are inverse bijections.
\end{proof}
\begin{definition}\label{def-ner-func-dirk}
    The nerve functor in the $k^{th}$ direction is defined as the composite
    \begin{equation*}
        \Nu{k}:\cat{n}(\clC)\xrw{\xi_k}\Cat(\cat{n-1}(\clC))\xrw{N}\funcat{}{\cat{n-1}(\clC)}
    \end{equation*}
    so that, in the above notation,
    \begin{equation*}
        (\Nu{k}X)_i=X\up{k}_i\qquad i=0,1\;.
    \end{equation*}
    Note that $\Nb{n}=\Nu{n}...\Nu{2}\Nu{1}$.
\end{definition}
\begin{notation}\label{not-ner-func-dirk}
    When $\clC=\Set$ we shall denote
    \begin{equation*}
        J_n=\Nu{n-1}\ldots \Nu{1}:\cat{n}\rw\funcat{n-1}{\Cat}\;.
    \end{equation*}
\end{notation}
 Thus $J_n$ amounts to taking the nerve construction in all but the last simplicial direction. The functor $J_n$ is fully faithful, thus we can identify $\cat{n}$ with the image $J_n(\cat{n})$ of the functor $J_n$.

 Given $X\in\cat{n}$, when no ambiguity arises we shall denote, for each $(s_1,\ldots ,s_{n-1})\in\Dmenop$
\begin{equation*}
    X_{s_1,\ldots ,s_{n-1}}=(J_n X)_{s_1,\ldots ,s_{n-1}}\in\Cat
\end{equation*}
and more generally, if $1\leq j \leq n-1$,
\begin{equation*}
    X_{s_1,\ldots ,s_{j}}=(\Nu{j}\ldots \Nu{1} X)_{s_1,\ldots ,s_{j}}\in\cat{n-j}\;.
\end{equation*}
Let $ob : \Cat \clC \rw \clC$ be the object of object functor. The left adjoint to $ob$ is the discrete internal category functor $d$. By Proposition \ref{pro-assoc-iso}  we then have
\begin{equation*}
\xymatrix{
    \cat{n}\clC \oset{\xi_n}{\cong}\Cat(\cat{n-1}\clC) \ar@<1ex>[r]^(0.65){ob} & \cat{n-1}\clC \ar@<1ex>[l]^(0.35){d}\;.
}
\end{equation*}

We denote
\begin{equation*}
\di{n}=\xi^{-1}_{n}\circ d \text{\;\;for\;\;} n>1,\; \di{1}=d \;.
\end{equation*}
Thus $\di{n}$ is the discrete inclusion of $\cat{n-1}\clC$ into $\cat{n}\clC$ in the $n^{th}$ direction.

\bk

The following is a characterization of objects of $\funcat{n-1}{\Cat}$ in the image of the functor $J_n$ in \ref{not-ner-func-dirk}.
\begin{lemma}\label{lem-char-obj}
     Let $L\in \funcat{n-1}{\Cat}$ be such that, for all $\uk\in\dop{n-1}$, $1\leq i\leq n-1$ and $k_i\geq 2$, the Segal maps are isomorphisms:
     \begin{equation}\label{eq-lem-char-obj}
        L_{\uk} \cong\pro{L_{\uk(1,i)}}{L_{\uk(0,i)}}{k_i}\;.
     \end{equation}
     Then $L\in\cat{n}$.
\end{lemma}
\begin{proof}
By induction on $n$. When $n=2$, $L\in\funcat{}{\Cat}$, $k\in \Dop$, $i=1$, $\uk(1,i)=1$, $\uk(0,i)=0$, $k_i=k=2$ and
\begin{equation*}
    L_k\cong\pro{L_1}{L_0}{k}\;.
\end{equation*}
Thus by Proposition \ref{pro-ner-int-cat}, $L\in\cat{2}$.

Suppose the lemma holds for $(n-1)$ and let $L \in \funcat{n-1}{\Cat}$ be as in the hypothesis. Consider $L_j\in\funcat{n-2}{\Cat}$ for $j\geq 0$. Let $\ur\in\dop{n-2}$ and denote $\uk=(j,\ur)\in\dop{n-1}$. Then, for any $2\leq i\leq n-1$, $k_i=r_{i-1}$ and
\begin{equation*}
    L_{\uk}=(L_j)_{\ur},\quad L_{\uk(1,i)}=(L_j)_{\ur(1,i-1)},\quad L_{\uk(0,i)}=(L_j)_{\ur(0,i-1)}\;.
\end{equation*}
Therefore \eqref{eq-lem-char-obj} implies
\begin{equation*}
    (L_j)_{\ur}=\pro{(L_j)_{\ur(1,i-1)}}{(L_j)_{\ur(0,i-1)}}{r_{i-1}}\;.
\end{equation*}
This means that $L_j$ satisfies the inductive hypothesis and therefore $L_j\in\cat{n-1}$.
Taking $i=1$ in \eqref{eq-lem-char-obj} we see that, for each $k_1\geq 2$ and $\ur=(k_2,\ldots,k_{n-1})\in\dop{n-2}$,
\begin{equation*}
    (L_{k_1})_{\ur}=\pro{(L_1)_{\ur}}{(L_0)_{\ur}}{k_1}\;.
\end{equation*}
That is, we have isomorphisms in $\cat{n-1}$
\begin{equation*}
    L_{k_1}\cong \pro{L_1}{L_0}{k_1}\;.
\end{equation*}
We conclude from Proposition \ref{pro-ner-int-cat} that $L\in\cat{n}$.

\end{proof}
\begin{lemma}\label{lem-char-obj-II}
    Let $P$ be the pullback in $\funcat{n-1}{\Cat}$ of the diagram in $A\rw C \lw B$. Suppose that for each $k\geq 2$ there are isomorphisms of Segal maps in $\funcat{n-2}{\Cat}$
    \begin{equation*}
        A_k\cong \pro{A_1}{A_0}{k},\;\;  C_k\cong \pro{C_1}{C_0}{k},\;\;  B_k\cong \pro{B_1}{B_0}{k}\;.
    \end{equation*}
    Then $P_k\cong \pro{P_1}{P_0}{k}$.
\end{lemma}
\begin{proof}
We show this for $k=2$, the case $k>2$ being similar. Since the nerve functor $N:\Cat\rw\funcat{}{\Set}$ commutes with pullbacks (as it is right adjoint) and pullbacks in $\funcat{n-1}{\Cat}$ are computed pointwise, for each $\us\in\dop{n-1}$ we have a pullback in $\Set$
\begin{equation*}
    \xymatrix{
    (NP)_{2\us} \ar[r] \ar[d] & (NA)_{2\us} \ar[d]\\
    (NC)_{2\us} \ar[r] & (NC)_{2\us}
    }
\end{equation*}
where
\begin{equation*}
    (NA)_{2\us}=\tens{(NA)_{1\us}}{(NA)_{0\us}}
\end{equation*}
and similarly for $NC$ and $NB$. We then calculate
\begin{align*}
    &(NP)_{2\us}=(NA)_{2\us} \tiund{(NC)_{2\us}} (NB)_{2\us} =\\
    & \resizebox{1.0\hsize}{!}{$
    =\{ \tens{(NA)_{1\us}}{(NA)_{0\us}} \}\tiund{\tens{(NC)_{1\us}}{(NC)_{0\us}}} \{ \tens{(NB)_{1\us}}{(NB)_{0\us}}\}\cong $}\\
    & \resizebox{1.0\hsize}{!}{$
    \cong\{ \tens{(NB)_{1\us}}{(NC)_{0\us}} \}\tiund{\tens{(NA)_{1\us}}{(NC)_{0\us}}} \{ \tens{(NB)_{1\us}}{(NC)_{0\us}}\}= $}\\
    &= \tens{(NP)_{1\us}}{(NP)_{0\us}}\;.
\end{align*}
In the above, the isomorphism before the last takes $(x_1,x_2,x_3,x_4)$ to $(x_1,x_3,x_2,x_4)$. Since this holds for all $\us$, it follows that
\begin{equation*}
    P_2\cong \tens{P_1}{P_0}\;.
\end{equation*}
The case $k>2$ is similar.

\end{proof}

\subsection{Some functors on $ \Cat$}\label{sbs-funct-cat}

The connected component functor
\begin{equation*}
    q: \Cat\rw \Set
\end{equation*}
associates to a category its set of paths components. This is left adjoint to the discrete category functor
\begin{equation*}
    \di{1}:\Set \rw \Cat
\end{equation*}
associating to a set $X$ the discrete category on that set. We denote by
\begin{equation*}
    \zgu{1}:\Id\Rw \di{1}q
\end{equation*}
the unit of the adjunction $q\dashv \di{1}$.
\begin{lemma}\label{lem-q-pres-fib-pro}
    $q$ preserves fiber products over discrete objects and sends
    equivalences of categories to isomorphisms.
\end{lemma}
\begin{proof}
We claim that $q$ preserves products; that is, given categories
$\clC$ and $\clD$, there is a bijection
\begin{equation*}
    q(\clC\times \clD)=q(\clC)\times q(\clD)\;.
\end{equation*}
In fact, given $(c,d)\in q(\clC\times \clD)$ the map
$q(\clC\times\clD)\rw q(\clC)\times q(\clD)$ given by
$[(c,d)]=([c],[d])$ is well defined and is clearly surjective. On
the other hand, this map is also injective: given $[(c,d)]$ and
$[(c',d')]$ with $[c]=[c']$ and $[d]=[d']$, we have paths in $\clC$
\newcommand{\lin}{-\!\!\!-\!\!\!-}
\begin{equation*}
\xymatrix @R5pt{c \hspace{2mm} \lin \hspace{2mm}\cdots \hspace{2mm}
\lin
\hspace{2mm}c'\\
d \hspace{2mm} \lin \hspace{2mm}\cdots \hspace{2mm} \lin
\hspace{2mm}d' }
\end{equation*}
and hence a path in $\clC\times \clD$
\begin{equation*}
(c,d)\hspace{2mm}\lin\hspace{2mm}\cdots\hspace{2mm}\lin\hspace{2mm}(c',d)
\hspace{2mm}\lin\hspace{2mm}\cdots\hspace{2mm}\lin\hspace{2mm}(c',d')\;.
\end{equation*}
Thus $[(c,d)]=[(c',d')]$ and so the map is also injective, hence it
is a bijection, as claimed.

Given a diagram in $\Cat$ $\xymatrix{\clC\ar_{f}[r] & \clE & \clD
\ar^{g}[l]}$ with $\clE$ discrete, we have
\begin{equation}\label{eq-q-pres-fib-pro}
    \clC\tiund{\clE}\clD=\underset{x\in\clE}{\coprod}\clC_x\times
    \clD_x
\end{equation}
where $\clC_x,\;\clD_x$ are the full subcategories of $\clC$ and
$\clD$ with objects $c,\;d$ such that \;$f(c)=x=g(d)$. Since $q$ preserves
products and (being left adjoint) coproducts, we conclude by
\eqref{eq-q-pres-fib-pro} that
\begin{equation*}
    q(\clC\tiund{\clE}\clD)\cong q(\clC)\tiund{\clE}\,q(\clD)\;.
\end{equation*}
Finally, if $F:\clC\simeq \clD:G$ is an equivalence of categories,
$FG\,\clC\cong\clC$ and $FG\,\clD\cong \clD$ which implies that
$qF\,qG\,\clC\cong q\clC$ and $qF\,qG\,\clD\cong q\clD$, so $q\clC$
and $q\clD$ are isomorphic.
\end{proof}
The isomorphism classes of objects functot
\begin{equation*}
    p:\Cat\rw\Set
\end{equation*}
associates to a category the set of isomorphism classes of its objects. Notice that if $\clC$ is a groupoid, $p\clC=q\clC$.
\begin{lemma}\label{lem-iso-cla-obj-fib-pro}\

    \begin{itemize}
      \item [a)] Let $X\xrw{f} Z \xlw{g}Y$ be a diagram in $\Cat$. Then
      \begin{equation*}
        p(X\tiund{Z}Y)\subseteq pX\tiund{pZ}pY\;.
      \end{equation*}

      \item [b)] Suppose, further, that $g_0=\Id$. Then
      \begin{equation*}
        p(X\tiund{Z}Y)\cong pX\tiund{pZ}pY\;.
      \end{equation*}
    \end{itemize}
\end{lemma}
\begin{proof}\

\nid a) The map
\begin{equation*}
    j:p(X\tiund{Z}Y)\rw pX\tiund{pZ}pY
\end{equation*}
is determined by the maps
\begin{equation*}
    p(X\tiund{Z}Y)\rw pX\quad \text{and}\quad p(X\tiund{Z}Y)\rw pY
\end{equation*}
induced by the projections
\begin{equation*}
    X\tiund{Z}Y \rw X \quad \text{and}\quad X\tiund{Z}Y \rw Y\;.
\end{equation*}
Thus, for each $(a,b)\in X\tiund{Z}Y$,
\begin{equation}\label{eq1-lem-iso-cla}
    j\,p(a,b)=(p(a),p(b))\;.
\end{equation}
Let $(a,b), (a',b')\in X\tiund{Z}Y$ be such that $p(a,b)=p(a',b')$. It follows by \eqref{eq1-lem-iso-cla} that $p(a)=p(a')$ and $p(b)=p(b')$. Thus there are isomorphisms $\za:a\cong a'$ in $X$ and $\zb: b\cong b'$ in $Y$ and in $Z$ we have
\begin{equation*}
    f\za=g\zb: fa=gb\cong fa'=gb'\;.
\end{equation*}
Thus $(\za,\zb):(a,b)\cong (a',b')$ is an isomorphism in $X\tiund{Z} Y$ and so $p(a,b) = p(a',b')$. This shows that $j$ is injective, proving a).
\mk

\nid b) By a) the map $j$ is injective. We now show that, if $g_0=\Id$, then $j$ is also surjective. Let $(px,py)\in pX\tiund{pZ} pY$ be such that $g_0=\Id$. Then
\begin{equation*}
    f p(x)= p f(x)=gp(y)=pg(y)=p(y)\;.
\end{equation*}
Also, $f(x)=g(f(x))$, so that $(x,f(x))\in X\tiund{Z}Y$. It follows that
\begin{equation*}
    j(x,f(x))=(p(x),pf(x))=(p(x),p(y))
\end{equation*}
so that $j$ is surjective. Hence $j$ is a bijection.

\end{proof}
\subsection{Pseudo-functors and their strictification}\label{sbs-pseudo-functors}
We recall the classical theory of strictification of pseudo-functors, see \cite{PW}, \cite{Lack}.

The functor 2-category $\funcat{n}{\Cat}$ is 2-monadic over $[ob(\Dnop),\Cat]$ where $ob(\Dnop)$ is the set of objects of $\Dnop$. Let
\begin{equation*}
    U:\funcat{n}{\Cat}\rw [ob(\Dnop),\Cat]
\end{equation*}
be the forgetful functor $(UX)_{\uk}=X_{\uk}$. Its left adjoint $F$ is given on objects by
\begin{equation*}
    (FH)_{\uk}=\underset{\ur\in ob(\Dnop)}{\coprod}\Dnop(\ur,\uk)\times H_{\ur}
\end{equation*}
for $H\in [ob(\Dmenop),\Cat]$, $\uk\in ob(\Dmenop)$. If $T$ is the monad corresponding to the adjunction $F\dashv U$, then
\begin{equation*}
    (TH)_{\uk}=\underset{\ur\in ob(\Dnop)}{\coprod}\Dnop(\ur,\uk)\times H_{\ur}
\end{equation*}
A pseudo $T$-algebra is given by $H\in [ob(\Dnop),\Cat]$,
\begin{equation*}
    h_{n}: \underset{\ur\in ob(\Dnop)}{\coprod}\Dnop(\ur,\uk)\times H_{\ur} \rw H_{\uk}
\end{equation*}
and additional data, as described in \cite{PW}. This amounts precisely to functors from $\Dnop$ to $\Cat$ and the 2-category $\sf{Ps\mi T\mi alg}$ of pseudo $T$-algebras corresponds to the 2-category $\Ps\funcat{n}{\Cat}$ of pseudo-functors, pseudo-natural transformations and modifications.

The strictification result proved in \cite{PW} yields that every pseudo-functor from $\Dnop$ to $\Cat$ is equivalent, in $\Ps\funcat{n}{\Cat}$, to a 2-functor.

Given a pseudo $T$-algebra as above, \cite{PW} consider the factorization  of $h:TH\rw H$ as
\begin{equation*}
    TH\xrw{v}L\xrw{g}H
\end{equation*}
with $v_{\uk}$ bijective on objects and $g_{\uk}$ fully faithful, for each $\uk\in\Dnop$. It is shown in \cite{PW} that it is possible to give a strict $T$-algebra structure $TL\rw L$ such that $(g,Tg)$ is an equivalence of pseudo $T$-algebras. It is immediate to see that, for each $\uk\in\Dnop$, $g_{\uk}$ is an equivalence of categories.

Further, it is shown in \cite{Lack} that $\St:\psc{n}{\Cat}\rw\funcat{n}{\Cat}$ as described above is left adjoint to the inclusion
\begin{equation*}
  J:\funcat{n}{\Cat}\rw\psc{n}{\Cat}
\end{equation*}
 and that the components of the units are equivalences in $\psc{n}{\Cat}$.
\subsection{Transport of structure}\label{transport-structure} We now recall a general categorical technique, known as transport of structure along an adjunction, with one of its applications. This will be used crucially in the proof of Theorem \ref{the-XXXX}.
\begin{theorem}\rm{\cite[Theorem 6.1]{lk}}\em\label{s2.the1}
    Given an equivalence $\;\eta,\;\zve : f \dashv f^* : A\rw B$ in the complete and locally small
    2-category $\clA$, and an algebra $(A,a)$ for the monad $T=(T,i,m)$ on $\clA$, the
    equivalence enriches to an equivalence
\begin{equation*}
  \eta,\zve:(f,\ovll{f})\vdash (f^*,\ovll{f^*}):(A,a)\rw(B,b,\hat{b},\ovl{b})
\end{equation*}
in $\PsTalg$, where $\hat{b}=\eta$, $\;\ovl{b}=f^* a\cdot T\zve \cdot
Ta\cdot T^2 f$, $\;\ovll{f}=\zve^{-1} a\cdot Tf$, $\;\ovll{f^*}=f^* a\cdot
T\zve$.
\end{theorem}
Let $\eta',\zve':f'\vdash f'^{*}:A'\rw B'$ be another equivalence in $\clA$ and
let $(B',b',\hat{b'},\ovl{b'})$ be the corresponding pseudo-$T$-algebra as in
Theorem \ref{s2.the1}. Suppose $g:(A,a)\rw(A',a')$ is a morphism in $\clA$ and
$\gamma$ is an invertible 2-cell in $\clA$
\begin{equation*}
  \xy
    0;/r.10pc/:
    (-20,20)*+{B}="1";
    (-20,-20)*+{B'}="2";
    (20,20)*+{A}="3";
    (20,-20)*+{A'}="4";
    {\ar_{f^*}"3";"1"};
    {\ar_{h}"1";"2"};
    {\ar^{f'^*}"4";"2"};
    {\ar^{g}"3";"4"};
    {\ar@{=>}^{\gamma}(0,3);(0,-3)};
\endxy
\end{equation*}
Let $\ovl{\gamma}$ be the invertible 2-cell given by the following pasting:
\begin{equation*}
    \xy
    0;/r.15pc/:
    (-40,40)*+{TB}="1";
    (40,40)*+{TB'}="2";
    (-40,-40)*+{B}="3";
    (40,-40)*+{B'}="4";
    (-20,20)*+{TA}="5";
    (20,20)*+{TA'}="6";
    (-20,-20)*+{A}="7";
    (20,-20)*+{A'}="8";
    {\ar^{Th}"1";"2"};
    {\ar_{b}"1";"3"};
    {\ar^{b'}"2";"4"};
    {\ar_{h}"3";"4"};
    {\ar_{Tg}"5";"6"};
    {\ar^{}"5";"7"};
    {\ar_{}"6";"8"};
    {\ar^{g}"7";"8"};
    {\ar_{Tf^*}"5";"1"};
    {\ar^{f^*}"7";"3"};
    {\ar^{Tf'^*}"6";"2"};
    {\ar^{f'^*}"8";"4"};
    {\ar@{=>}^{(T\gamma)^{-1}}(0,33);(0,27)};
    {\ar@{=>}^{\gamma}(0,-27);(0,-33)};
    {\ar@{=>}^{\ovll{f'^*}}(30,3);(30,-3)};
    {\ar@{=>}^{\ovll{f^*}}(-30,3);(-30,-3)};
\endxy
\end{equation*}
Then it is not difficult to show that
$(h,\ovl{\gamma}):(B,b,\hat{b},\ovl{b})\rw(B',b',\hat{b'},\ovl{b'})$ is a
pseudo-$T$-algebra morphism.

The following fact is essentially known and, as sketched in the proof below, it is an instance of Theorem \ref{s2.the1}
\begin{lemma}\cite{PP}\label{lem-PP}
     Let $\clC$ be a small 2-category, $F,F':\clC\rw\Cat$ be 2-functors, $\alpha:F\rw F'$
    a 2-natural transformation. Suppose that, for all objects $C$ of $\clC$, the
    following conditions hold:
\begin{itemize}
  \item [i)] $G(C),\;G'(C)$ are objects of $\Cat$ and there are adjoint equivalences of
  categories $\mu_C\vdash\eta_C$, $\mu'_C\vdash\eta'_C$,
\begin{equation*}
  \mu_C:G(C)\;\rightleftarrows\;F(C):\eta_C\qquad\qquad
  \mu'_C:G'(C)\;\rightleftarrows\;F'(C):\eta'_C,
\end{equation*}
  \item [ii)] there are functors $\beta_C:G(C)\rw G'(C),$
  \item [iii)] there is an invertible 2-cell
\begin{equation*}
  \gamma_C:\beta_C\,\eta_C\Rightarrow\eta'_C\,\alpha_C.
\end{equation*}
\end{itemize}
Then
\begin{itemize}
  \item [a)] There exists a pseudo-functor $G:\clC\rw\Cat$ given on objects by $G(C)$,
  and pseudo-natural transformations $\eta:F\rw G$, $\mu:G\rw F$ with
  $\eta(C)=\eta_C$, $\mu(C)=\mu_C$; these are part of an adjoint equivalence
  $\mu\vdash\eta$ in the 2-category $\Ps[\clC,\Cat]$.
  \item [b)] There is a pseudo-natural transformation $\beta:G\rw G'$ with
  $\beta(C)=\beta_C$ and an invertible 2-cell in $\Ps[\clC,\Cat]$,
  $\gamma:\beta\eta\Rightarrow\eta\alpha$ with $\gamma(C)=\gamma_C$.
\end{itemize}
\end{lemma}
\begin{proof}
Recall \cite{PW} that the functor 2-category $[\clC,\Cat]$ is 2-monadic over
$[ob(\clC),\Cat]$, where $ob(\clC)$ is the set of objects in $\clC$. Let
\begin{equation*}
  \clU:[\clC,\Cat]\rw[ob(\clC),\Cat]
\end{equation*}
be the forgetful functor. Let $T$ be the 2-monad; then the pseudo-$T$-algebras are precisely the pseudo-functors from
$\clC$ to $\Cat$.

Then the adjoint equivalences $\mu_C\vdash\eta_C$ amount precisely to an
adjoint equivalence in $[ob(\clC),\Cat]$, $\;\mu_0\vdash\eta_0$,
$\;\mu_0:G_0\;\;\rightleftarrows\;\;\clU F:\eta_0$ where $\;G_0(C)=G(C)$ for
all $C\in ob(\clC)$. This equivalence enriches to an
adjoint equivalence $\mu\vdash\eta$ in $\Ps[\clC,\Cat]$
\begin{equation*}
  \mu:G\;\rightleftarrows\; F:\eta
\end{equation*}
between $F$ and a pseudo-functor $G$; it is $\clU G=G_0$, $\;\clU\eta=\eta_0$,
$\;\clU\mu=\mu_0$; hence on objects $G$ is given by $G(C)$, and
$\eta(C)=\clU\eta(C)=\eta_C$, $\;\mu(C)=\clU\mu(C)=\mu_C$.

Let $\nu_C:\Id_{G(C)}\Rw\eta_C\mu_C$ and $\zve_C:\mu_C\eta_C\Rw\Id_{F(C)}$ be
the unit and counit of the adjunction $\mu_C\vdash\eta_C$. Given a morphism $f:C\rw D$ in $\clC$, it is
\begin{equation*}
  G(f)=\eta_D F(f)\mu_C
\end{equation*}
and we have natural isomorphisms:
\begin{align*}
   & \eta_f : G(f)\eta_C=\eta_D F(f)\mu_C\eta_C\overset{\eta_D F(f)\zve_C}{=\!=\!=\!=\!\Rw} \eta_D F(f)\\
   & \mu_f : F(f)\mu_C\overset{\nu_{F(f)}\mu_C}{=\!=\!=\!\Rw}\mu_D\eta_D F(f)\mu_C=\mu_D
   G(f).
\end{align*}
Also, the natural isomorphism
\begin{equation*}
  \beta_f: G'(f)\beta_C\Rw\beta_D G(f)
\end{equation*}
is the result of the following pasting
\begin{equation*}
    \xy
    0;/r.15pc/:
    (-40,40)*+{G(C)}="1";
    (40,40)*+{G'(C)}="2";
    (-40,-40)*+{G(D)}="3";
    (40,-40)*+{G'(D)}="4";
    (-20,20)*+{F(C)}="5";
    (20,20)*+{F'(C)}="6";
    (-20,-20)*+{F(D)}="7";
    (20,-20)*+{F'(D)}="8";
    {\ar^{\beta_C}"1";"2"};
    {\ar_{G(f)}"1";"3"};
    {\ar^{G'(f)}"2";"4"};
    {\ar_{\beta_D}"3";"4"};
    {\ar^{\alpha_C}"5";"6"};
    {\ar^{F(f)}"5";"7"};
    {\ar_{F'(f)}"6";"8"};
    {\ar_{\alpha'_D}"7";"8"};
    {\ar^{}"5";"1"};
    {\ar^{}"7";"3"};
    {\ar^{}"6";"2"};
    {\ar^{}"8";"4"};
    {\ar@{=>}^{\gamma_C}(0,33);(0,27)};
    {\ar@{=>}^{\gamma_D^{-1}}(0,-27);(0,-33)};
    {\ar@{=>}^{\eta_f'}(30,3);(30,-3)};
    {\ar@{=>}^{\eta_f}(-30,3);(-30,-3)};
\endxy
\end{equation*}
\end{proof}
\section{Weakly globular $n$-fold categories}\label{sec-WG-nfol-cat}

In this section we review the notions of weakly globular \nfol category and of Segalic pseudo-functors introduced by the author in \cite{Pa2}. These will be used throughout the paper.

\subsection{Homotopically discrete $n$-fold categories}\label{homotopically-discrete}  We first recall the category of homotopically discrete \nfol categories, introduced by the author in \cite{Pa1}. This is needed to define weakly globular $n$-fold categories.
\begin{definition}\label{def-hom-dis-ncat}
    Define inductively the full subcategory $\cathd{n}\subset\cat{n}$ of homotopically discrete \nfol categories.

    For $n=1$, $\cathd{1}=\cathd{}$ is the category of  equivalence relations. Denote by $\p{1}=p:\Cat\rw\Set$ the isomorphism classes of objects functor.

    Suppose, inductively, that for each $1\leq k\leq n-1$ we defined $\cathd{k}\subset\cat{k}$ and $k$-equivalences such that the following holds:
    \begin{itemize}
      \item [a)] The $k^{th}$ direction in $\cathd{k}$ is groupoidal; that is, if $X\in\cathd{k}$, $\xi_{k}X\in\Gpd(\cat{k-1})$ (where $\xi_{k}X$ is as in Proposition \ref{pro-assoc-iso}).

      \item [b)] There is a functor $\p{k}:\cathd{k}\rw\cathd{k-1}$ making the following diagram commute:
          \begin{equation}\label{eq1-p-fun-def}
            \xymatrix{
            \cathd{k} \ar^{\Nu{k-1}...\Nu{1}}[rrr] \ar_{p^{(k)}}[d] &&& \funcat{k-1}{\Cat} \ar^{\bar p}[d]\\
            \cathd{k-1} \ar_{\N{k-1}}[rrr] &&& \funcat{k-1}{\Set}
            }
          \end{equation}
          Note that this implies that $(\p{k}X)_{s_1 ... s_{k-1}}=p X_{s_1 ... s_{k-1}}$ for all $(s_1 ... s_{k-1})\in\dop{k-1}$.

    \end{itemize}

   $\cathd{n}$ is the full subcategory of $\funcat{}{\cathd{n-1}}$ whose objects $X$ are such that
    \bigskip
    \begin{itemize}
      \item [(i)]  $\hspace{30mm} X_s\cong\pro{X_1}{X_0}{s} \quad \mbox{for all} \; s \geq 2.$

\medskip
    In particular this implies that $X\in \Cat(\Gpd(\cat{n-2})) =\Gpd(\cat{n-1})$ and the $n^{th}$ direction in $X$ is groupoidal.
    \medskip
      \item [(ii)] The functor
      \begin{equation*}
        \op{n-1}:\cathd{n}\subset \funcat{}{\cathd{n-1}}\rw\funcat{}{\cathd{n-2}}
      \end{equation*}
      restricts to a functor
      \begin{equation*}
        \p{n}:\cathd{n}\rw\cathd{n-1}
      \end{equation*}
     Note that this implies that $(\p{n}X)_{s_1 ... s_{n-1}}=p X_{s_1 ... s_{n-1}}$ and that the following diagram commutes
          \begin{equation}\label{eq2-p-fun-def}
            \xymatrix{
            \cathd{n} \ar^{\Nu{n-1}...\Nu{1}}[rrr] \ar_{p^{(n)}}[d] &&& \funcat{n-1}{\Cat} \ar^{\bar p}[d]\\
            \cathd{n-1} \ar_{\N{n-1}}[rrr] &&& \funcat{n-1}{\Set}
            }
          \end{equation}
     \end{itemize}
\end{definition}
\mk
\begin{definition}\label{def-hom-dis-ncat-0}

    Denote by $\zgu{n}_X:X\rw \di{n}\p{n}X$ the morphism given by
    \begin{equation*}
        (\zgu{n}_X)_{s_1...s_{n-1}} :X_{s_1...s_{n-1}} \rw d p X_{s_1...s_{n-1}}
    \end{equation*}
    for all  $(s_1,...,s_{n-1})\in \dop{n-1}$. Denote by
    \begin{equation*}
        X^d =\di{n}\di{n-1}...\di{1}\p{1}\p{2}...\p{n}X
    \end{equation*}
    and by $\zg\lo{n}$ the composite
    \begin{equation*}
        X\xrw{\zgu{n}}\di{n}\p{n}X \xrw{\di{n}\zgu{n-1}} \di{n}\di{n-1}\p{n-1}\p{n}X \rw \cdots \rw X^d\;.
    \end{equation*}
    For each $a,b\in X_0^d$ denote by $X(a,b)$ the fiber at $(a,b)$ of the map
    \begin{equation*}
        X_1 \xrw{(d_0,d_1)} X_0\times X_0 \xrw{\zg\lo{n}\times\zg\lo{n}} X_0^d\times X_0^d\;.
    \end{equation*}
\end{definition}
\begin{definition}\label{def-hom-dis-ncat-1}
Define inductively $n$-equivalences in $\cathd{n}$. For $n=1$, a 1-equivalence is an equivalence of categories. Suppose we defined $\nm$-equivalences in $\cathd{n-1}$. Then a map $f:X\rw Y$ in $\cathd{n}$ is an $n$-equivalence if, for all $a,b \in X_0^d$, $f(a,b):X(a,b) \rw Y(fa,fb)$ and $\p{n}f$ are $\nm$-equivalences.
\end{definition}
\subsection{Weakly globular $n$-fold categories}\label{weakly-globular} We recall the category of weakly globular \nfol categories and some of its main properties needed in this paper.

\begin{definition}\label{def-n-equiv}
    For $n=1$, $\catwg{1}=\Cat$ and $1$-equivalences are equivalences of categories.

    Suppose, inductively, that we defined $\catwg{n-1}$ and $(n-1)$-equivalences. Then $\catwg{n}$ is the full subcategory of $\funcat{}{\catwg{n-1}}$ whose objects $X$ are such that
    \begin{itemize}
      \item [a)] \textsl{Weak globularity condition} $X_0\in\cathd{n-1}$.\mk
      \item [b)] \textsl{Segal condition} For all $k\geq 2$ the Segal maps are isomorphisms:
      \begin{equation*}
        X_k\cong\pro{X_1}{X_0}{k}\;.
      \end{equation*}

      \item [c)] \textsl{Induced Segal condition} For all $k\geq 2$ the induced Segal maps
      \begin{equation*}
        X_k\rw\pro{X_1}{X^d_0}{k}
      \end{equation*}
      (induced by the map $\zg:X_0\rw X_0^d$) are $(n-1)$-equivalences.\mk

      \item [d)] \textsl{Truncation functor} There is a functor $\p{n}:\catwg{n}\rw\catwg{n-1}$ making the following diagram commute
      \begin{equation*}
        \xymatrix{
        \catwg{n} \ar^{J_n}[rr] \ar_{\p{n}}[d] && \funcat{n-1}{\Cat} \ar^{\ovl p}[d]\\
        \catwg{n-1} \ar^{\Nb{n-1}}[rr]  && \funcat{n-1}{\Set}
        }
    \end{equation*}
    \end{itemize}
    Given $a,b\in X_0^d$, denote by $X(a,b)$ the fiber at $(a,b)$ of the map
    \begin{equation*}
         X_1\xrw{(\pt_0,\pt_1)} X_0\times X_0 \xrw{\zg\times \zg}  X^d_0\times X^d_0\;.
    \end{equation*}
    We say that a map $f:X\rw Y$ in $\catwg{n}$ is an $n$-equivalence if
    \begin{itemize}
      \item [i)] For all $a,b\in X_0^d$
      \begin{equation*}
        f(a,b): X(a,b) \rw Y(fa,fb)
      \end{equation*}
      is an $(n-1)$-equivalence.\mk

      \item [ii)] $\p{n}f$ is an $(n-1)$-equivalence.
    \end{itemize}
    This completes the inductive step in the definition of $\catwg{n}$.
\end{definition}
\begin{remark}\label{rem-n-equiv}
    It follows by Definition \ref{def-n-equiv}, Definition \ref{def-hom-dis-ncat} and \cite[Proposition 3.4]{Pa2} that $\cathd{n}\subset \catwg{n}$.
\end{remark}
\begin{definition}\label{def-kdir-wg-ncat}
Given $X\in\cat{n}$ and $k\geq 0$, let $\Nu{2}X\in\funcat{}{\cat{n-1}}$ as in Definition \ref{def-ner-func-dirk} so that for each $k\geq 0$, $(\Nu{2}X)_k=X_k\up{2}\in \funcat{}{\cat{n-2}}$ is given by
    \begin{equation*}
        (X_k\up{2})_s =
        \left\{
          \begin{array}{ll}
            X_{0k}, & \hbox{$s=0$;} \\
            X_{1k}, & \hbox{$s=1$;} \\
            X_{sk}=\pro{X_{1k}}{X_{0k}}{s}, & \hbox{$s\geq 2$.}
          \end{array}
        \right.
    \end{equation*}
\end{definition}
The following property of weakly globular \nfol categories will be used in Section \ref{sec-cat-lta}.
\begin{proposition}\label{pro-property-WG-nfol}
    The functor $\Nu{2}:\cat{n}\rw \funcat{}{\cat{n-1}}$ restricts to a functor
    \begin{equation*}
        \Nu{2}:\catwg{n} \rw  \funcat{}{\catwg{n-1}}\;.
    \end{equation*}
\end{proposition}
\begin{proof}
See \cite[Proposition 3.16 a)]{Pa2}
\end{proof}

\begin{definition}\label{def-pn}
    For each $1\leq j \leq n$ denote
    \begin{align*}
        \p{j,n} & = \p{j}\p{j-1}\cdots \p{n}:\catwg{n}\rw \catwg{j-1}\\
        \p{n,n}& = \p{n}\;.
    \end{align*}
\end{definition}
\begin{lemma}\label{lem-prop-pn}\cite [Lemma 3.8]{Pa2}
    For each $X\in\catwg{n}$, $1\leq j < n$ and $s\geq 2$ it is
    \begin{equation}\label{eq-lem-prop-pn}
    \begin{split}
        &\p{j,n-1}X_s \cong  \p{j,n-1}(\pro{X_1}{X_0}{s})=\\
         & =\pro{\p{j,n-1} X_1}{\p{j,n-1} X_0}{s}\;.
    \end{split}
    \end{equation}
\end{lemma}
\begin{remark}\label{rem-eq-def-wg-ncat}
    It follows immediately from Lemma \ref{lem-prop-pn} that if $X\in\catwg{n}$, for all $s\geq 2$
    \begin{equation}\label{eq1-rem-eq-def-wg-ncat}
        X_{s0}^d=(\pro{X_{10}}{X_{00}}{s})^d\cong\pro{X^d_{10}}{X^d_{00}}{s}\;.
    \end{equation}
    In fact, by \eqref{eq-lem-prop-pn} in the case $j=2$, taking the 0-component, we obtain
    \begin{equation*}
    \begin{split}
        & \p{1,n-2}(\pro{X_{10}}{X_{00}}{s})\cong \\
        =\ & \pro{\p{1,n-2}X_{10}}{\p{1,n-2}X_{00}}{s}
    \end{split}
    \end{equation*}
    which is the same as \eqref{eq1-rem-eq-def-wg-ncat}.
\end{remark}
\subsection{Segalic pseudo-functors}\label{segalic-pseudo} We now recall the notion of Segalic pseudo-functor from \cite{Pa2}.

Let $H\in\Ps\funcat{n}{\Cat}$ be such that $H_{\uk(0,i)}$ is discrete for all $\uk\in\Dmenop$ and all $i\geq 0$. Then the following diagram commutes, for each $k_i\geq 2$.
\begin{equation*}
    \xy
    0;/r.8pc/:
    (0,0)*+{H_{\uk}}="1";
    (-7,-5)*+{H_{\uk(1,i)}}="2";
    (-2,-5)*+{H_{\uk(1,i)}}="3";
    (9,-5)*+{H_{\uk(1,i)}}="4";
    (-10,-9)*+{H_{\uk(0,i)}}="5";
    (-5,-9)*+{H_{\uk(0,i)}}="6";
    (0,-9)*+{H_{\uk(0,i)}}="7";
    (6,-9)*+{H_{\uk(0,i)}}="8";
    (12,-9)*+{H_{\uk(0,i)}}="9";
    (3,-5)*+{\cdots}="10";
    (3,-9)*+{\cdots}="11";
    {\ar_{\nu_1}"1";"2"};
    {\ar^{\nu_2}"1";"3"};
    {\ar^{\nu_k}"1";"4"};
    {\ar_{d_1}"2";"5"};
    {\ar^{d_0}"2";"6"};
    {\ar^{d_1}"3";"6"};
    {\ar^{d_0}"3";"7"};
    {\ar_{d_1}"4";"8"};
    {\ar^{d_0}"4";"9"};
    \endxy
\end{equation*}
There is therefore a unique Segal map
\begin{equation*}
    H_{\uk}\rw \pro{H_{\uk(1,i)}}{H_{\uk(0,i)}}{k_i}\;.
\end{equation*}
\begin{definition}\label{def-seg-ps-fun}
    We\; define\; the \;subcategory \; $\segpsc{n}{\Cat}$\; of $\psc{n}{\Cat}$ as follows:

    For $n=1$, $H\in \segpsc{}{\Cat}$ if $H_0$ is discrete and the Segal maps are isomorphisms: that is, for all $k\geq 2$
    \begin{equation*}
        H_k\cong\pro{H_1}{H_0}{k}
    \end{equation*}
    Note that, since $p$ commutes with pullbacks over discrete objects, there is a functor
   \begin{equation*}
   \begin{split}
       & \p{2}:\segpsc{}{\Cat} \rw \Cat\;, \\
       & (\p{2}X)_{k}=p X_k\;.
   \end{split}
   \end{equation*}
   That is the following diagram commutes:
   \begin{equation*}
    \xymatrix{
    \segpsc{}{\Cat} \ar@{^{(}->}[rr]\ar_{\p{2}}[d] && \psc{}{\Cat} \ar^{\ovl{p}}[d]\\
    \Cat \ar[rr] && \funcat{}{\Set}
    }
   \end{equation*}
   When $n>1$, $\segpsc{n}{\Cat}$ is the full subcategory of $\psc{n}{\Cat}$ whose objects $H$ satisfy the following: \mk
   \begin{itemize}
     \item [a)] \emph{Discreteness condition}:  $H_{\uk(0,i)}$ is discrete for all $\uk\in\Dmenop$ and $1 \leq i \leq n$.\mk
     \item [b)] \emph{Segal condition}: All Segal maps are isomorphisms
      \begin{equation*}
      H_{\uk}\cong\pro{H_{\uk(1,i)}}{H_{\uk(0,i)}}{k_i}
      \end{equation*}
      for all $\uk\in\Dmenop$, $1 \leq i \leq n$ and $k_i\geq2$.\mk
     \item [c)] \emph{Truncation functor}: There is a functor
     \begin{equation*}
        \p{n+1}:\segpsc{n}{\Cat}\rw \catwg{n}
     \end{equation*}
   \end{itemize}
    making the following diagram commute:
    \begin{equation*}
        \xymatrix{
        \segpsc{n}{\Cat} \ar@{^(->}[r] \ar_{\p{n+1}}[d] & \psc{n}{\Cat} \ar^{\ovl{p}}[d]\\
        \catwg{n} \ar_{\Nb{n}}[r] & \funcat{n}{\Set}
        }
    \end{equation*}
\end{definition}
\bigskip

The main property of Segalic pseudo-functors is the theorem below stating that that the classical strictification of pseudo-functors, when applied to a Segalic pseudo-functors, yields weakly globular \nfol categories. This result is used crucially in Section \ref{sec-cat-lta} to build the rigidification functor $Q_n$.

\begin{theorem}\cite[Theorem 4.5]{Pa2}\label{the-strict-funct}
    The strictification functor
    \begin{equation*}
        \St  : \psc{n-1}{\Cat}\rw \funcat{n-1}{\Cat}
    \end{equation*}
    restricts to a functor
    \begin{equation*}
        L_n: \segpsc{n-1}{\Cat}\rw J_n\catwg{n}
    \end{equation*}
    where $J_n\catwg{n}$ denotes the image of the fully faithful functor $J_n:\catwg{n}\rw\funcat{n-1}{\Cat}$\;.
    Further, for each $H\in \segpsc{n-1}{\Cat}$ and $\uk\in\dop{n-1}$, the map $(L_n H)_{\uk}\rw H_{\uk}$ is an equivalence of categories.
\end{theorem}

\section{Weakly globular Tamsamani n-categories}\label{sec-WG-Tam-cat}
In this  section we introduce the category $\tawg{n}$ of weakly globular Tamsamani $n$-categories and discuss their properties. The definition is inductive on dimension starting with $\Cat$ when $n=1$.

 In dimension $n>1$, a weakly globular Tamsamani $n$-category $X$ is a simplicial object in the category $\tawg{n-1}$ satisfying additional conditions. These conditions encode the weakness in the structure in two ways: one is the weak globularity condition, requiring $X_{0}$ to be a homotopically discrete $\nm$-fold category. The second is the induced Segal maps condition, requiring that the induced Segal maps for all $k\geq 2$
\begin{equation}\label{eq1-sec-prelim}
  \hmu{k}:X_k \rw \pro{X_1}{X_0^d}{k}
\end{equation}
are $\nm$-equivalences.

By unravelling the inductive Definition \ref{def-wg-ps-cat} we obtain the embedding
\begin{equation*}
  J_n : \tawg{n}\rw \funcat{n-1}{\Cat}\;.
\end{equation*}
The inductive Definition \ref{def-wg-ps-cat} also requires the existence of a truncation functor
\begin{equation*}
  \p{n}:\tawg{n}\rw \tawg{n-1}
\end{equation*}
obtained by applying levelwise to $J_n X$ the isomorphism classes of object functor $p:\Cat\rw\Set$.

 This truncation functor is used  to define $\nequ$s in $\tawg{n}$. This notion is a higher dimensional generalization of a functor which is fully faithful and essentially surjective on objects. A useful characterization of $n$-equivalences is given in Proposition \ref{pro-n-equiv}.

Under certain conditions, an $\nequ$ $f$ is such that $J_n f$ is a levelwise equivalence of categories (see Proposition \ref{pro-crit-lev-nequiv}) and this will be used in Section \ref{sec-func-qn}.
\begin{definition}\rm\label{def-wg-ps-cat}
    We define the category $\tawg{n}$ by induction on $n$. For $n=1$, $\tawg{1}=\Cat$ and 1-equivalences are equivalences of categories. We denote by $\p{1}=p:\Cat\rw \Set$ the isomorphism classes of object functor.

    Suppose, inductively, that we defined for each $1 < k\leq n-1$
    \begin{equation*}
        \xymatrix{\tawg{k}\;\ar@{^{(}->}^(0.4){}[r] & \;\funcat{k-1}{\Cat}}
    \end{equation*}
    and $k$-equivalences in $\tawg{k}$ as well as a functor
    \begin{equation*}
        \p{k}:\tawg{k}\rw \tawg{k-1}
    \end{equation*}
    sending $k$-equivalences to $(k-1)$-equivalences and making the following diagram commute:
    \begin{equation}\label{eq-wg-ps-cat}
    \xymatrix@C=30pt{
    \tawg{k} \ar^{J_{k}}[rr]\ar_{\p{k}}[d] && \funcat{k-1}{\Cat} \ar^{\ovl{p}}[d]\\
    \tawg{k-1}  \ar^{\Nb{k-1}}[rr] && \funcat{k-1}{\Set}
    }
    \end{equation}
    Define $\tawg{n}$ to be the full subcategory of $\funcat{}{\tawg{n-1}}$ whose objects $X$ are such that
        \begin{itemize}
      \item [a)] \textsl{Weak globularity condition} \;$X_0\in\cathd{n-1}$.\bk

      \item [b)] \textsl{Induced Segal maps condition}. For all $s\geq 2$ the induced Segal maps
      \begin{equation*}
       X_s  \rw \pro{X_1}{X^d_0}{s}
      \end{equation*}
      (induced by the map $\zg:X_0\rw X_0^d$) are $(n-1)$-equivalences.
    \end{itemize}
\medskip
To complete the inductive step, we need to define $\p{n}$ and $n$-equivalences.
Note that the functor
\begin{equation*}
    \ovl{\p{n-1}}:\funcat{}{\tawg{n-1}}\rw \funcat{}{\tawg{n-2}}
\end{equation*}
restricts to a functor
\begin{equation*}
    \p{n}:\tawg{n}\rw \tawg{n-1}\;.
\end{equation*}
In fact, by \eqref{eq-wg-ps-cat} $\p{n-1}$ preserves pullbacks over discrete objects so that
\begin{equation*}
  \p{n-1}(\pro{X_1}{X_0^d}{s}) \cong\pro{\p{n-1}X_1}{(\p{n-1}X_0^d)}{s}\; .
\end{equation*}
 Further, $\p{n-1}X_0^d = (\p{n-1}X_0)^d$ and $\p{n-1}$ sends $(n-1)$-equivalences to $(n-2)$-equivalences.

 Therefore, the induced Segal maps for $s\geq 2$
\begin{equation*}
    X_s \rw\pro{X_1}{X_0^d}{s}
\end{equation*}
being $(n-1)$-equivalences, give rise to $(n-2)$-equivalences
\begin{equation*}
     \p{n-1}X_s \rw \pro{\p{n-1}X_1}{(\p{n-1}X_0)^d}{s}\; .
\end{equation*}
This shows that $\p{n}X \in \tawg{n-1}$. It is immediate that \eqref{eq-wg-ps-cat} holds at step $n$.

Given $a,b\in X_0^d$, denote by $X(a,b)$ the fiber at $(a,b)$ of the map
    \begin{equation*}
         X_1\xrw{(\pt_0,\pt_1)} X_0\times X_0 \xrw{\zg\times \zg}  X^d_0\times X^d_0\;.
    \end{equation*}
    We say that a map $f:X\rw Y$ in $\tawg{n}$ is an $n$-equivalence if
    \begin{itemize}
      \item [i)] For all $a,b\in X_0^d$
      \begin{equation*}
        f(a,b): X(a,b) \rw Y(fa,fb)
      \end{equation*}
      is an $(n-1)$-equivalence.\mk

      \item [ii)] $\p{n}f$ is an $(n-1)$-equivalence.
    \end{itemize}
    This completes the inductive step in the definition of $\tawg{n}$.

\end{definition}
\begin{remark}\label{rem-wg-ps-cat}
    It follows by Definition \ref{def-n-equiv} that $\catwg{n}\subset\tawg{n}$.
\end{remark}
\begin{definition}\label{def-x-tawg-disc}
    An object $X\in\tawg{n}$ is called discrete if $\Nb{n-1}X$ is a constant functor.
\end{definition}
\begin{example}\label{ex-tam}
Tamsamani $n$-categories.
\mk

A special case of weakly globular Tamsamani $n$-category occurs when $X\in\tawg{n}$ is such that $X_0$ and $X_{\oset{r}{1...1}0}$ are discrete for all $1 \leq r \leq n-2$. The resulting category is the category $\Tan$ of Tamsamani's $n$-categories. Note that, if $X\in\Tan$ then $X_s \in \ta{n-1}$ for all $n$, the induced Segal maps $\hmu{s}$ coincide with the Segal maps
\begin{equation*}
  \nu_s:X_s\rw \pro{X_1}{X_0}{s}
\end{equation*}
and $\p{n}X\in\ta{n-1}$. Hence this recovers the original definition of Tamsamani's weak $n$-category \cite{Ta}.

\end{example}
\begin{example}\label{ex-wg-2-3}
    Weakly globular Tamsamani 2-categories.
    \mk

    \nid From the definition, $X\in \tawg{2}$ consists of a simplicial object $X\in\funcat{}{\Cat}$ such that $X_0\in\cathd{}$ and the induced Segal maps
    \begin{equation*}
        \hmuk:X_k \xrw{\muk}\pro{X_1}{X_0}{k}\xrw{\nu_k}\pro{X_1}{X^d_0}{k}
    \end{equation*}
    are equivalences of categories.

    The existence of the functor $\p{2}:\tawg{2}\rw \Cat$ making diagram \eqref{eq-wg-ps-cat} commute is in this case automatic so this condition does need to be imposed as part of the definition. In fact, since $p:\Cat\rw \Set$ sends equivalences to isomorphisms and preserves pullbacks over discrete objects, we obtain isomorphisms for each $k\geq 2$
    \begin{flalign*}
        p\hmuk: p X_k & \cong p(\pro{X_1}{X_0^d}{k})\cong &\\
        & \cong \pro{p X_1}{p X_{0}^{d}}{k}\cong \pro{p X_1}{(p X_{0})^{d}}{k}\;.&
    \end{flalign*}
    Since $p X_k =(\p{2}X)_k$ for all $k$,  the Segal maps for the simplicial set $\bar p X$ are isomorphisms. Therefore $\bar p X$ is the nerve of a category. We therefore have a functor $\p{2}$ given as composite
    \begin{equation*}
        \p{2}:\tawg{2}\xrw{\bar p} \mbox{\sf{Ner}} \Cat \subset \funcat{}{\Set}\xrw{P}\Cat
    \end{equation*}
    where $\mbox{\sf{Ner}} \Cat$ is the full subcategory of $\funcat{}{\Set}$ consisting of nerves of categories and $P$ is the left adjoint to the nerve functor $N:\Cat\rw\funcat{}{\Set}$.
    \end{example}
\begin{notation}\label{not-ex-tam}
    For each $1 \leq j\leq n-1$ denote\mk

      $\begin{array}{ll}
        \p{j,n}=\p{j}\p{j+1}...\p{n}:\tawg{n}\rw \tawg{j-1},\\
        \p{n,n}=\p{n}.
      \end{array}$
\end{notation}
In the following lemma we give a criterion for a weakly globular Tamsamani $n$-category to be in $\catwg{n}$.
\begin{lemma}\label{lem-x-in-tawg-x-in-catwg}
  Let $X\in\tawg{n}$ be such that
      \begin{itemize}
      \item [a)] $X_s\in \catwg{n-1}$ for all $s$.\mk

      \item [b)] $X_s \cong \pro{X_1}{X_0}{s}$ for all $s\geq 2$.\mk

      \item [c)] For all $s\geq 2$ and $1\leq j\leq n-1$
          \begin{equation}\label{eq-lem-eq-def-wg}
          \begin{split}
              & \p{j,n-1}X_s \cong  \p{j,n-1}(\pro{X_1}{X_0}{s})= \\
             = \ &\pro{\p{j,n-1} X_1}{\p{j,n-1} X_0}{s}
          \end{split}
          \end{equation}
    \end{itemize}
    then $X\in\catwg{n}$.
\end{lemma}
\begin{proof}
By induction on $n$. When $n=2$, let $X\in\tawg{2}$ satisfy a), b), c). Then $X_0\in\cathd{}$ and by b), $X\in\cat{2}$. Since $X\in\tawg{2}$ the induced Segal maps
\begin{equation*}
  X_s \rw \pro{X_1}{X_0^d}{s}
\end{equation*}
are equivalences of categories for all $s\geq 2$. Thus, by definition, $X\in\catwg{2}$.

Suppose, inductively, that the lemma holds for $(n-1)$ and let $X\in\tawg{n}$ satisfy a), b), c). Then $X_0\in\cathd{n-1}$, $X_k\in\catwg{n-1}$ for all $k\geq 0$ and, since $X\in\tawg{n}$, the induced Segal maps
\begin{equation*}
  X_s \rw \pro{X_1}{X_0^d}{s}
\end{equation*}
are $(n-1)$-equivalences for all $s\geq 2$.

By hypothesis b) and the definition of $\catwg{n}$, to show that $X\in\catwg{n}$ it is enough to prove that $\p{n}X\in\catwg{n-1}$. We do so by proving that $\p{n-1}X$ satisfies the inductive hypothesis.

For all $s\geq 0$, $(\p{n}X)_s=\p{n-1}X_s\in\catwg{n-2}$ since, by hypothesis a), $X_s\in\catwg{n-1}$. Thus $\p{n}X$ satisfies the inductive hypothesis a). Also, from hypothesis c) for $X$ in the case $j=n-1$  for all $s\geq 2$,
\begin{equation*}
\begin{split}
    & \p{n-1}X_s \cong\p{n-1}(\pro{X_1}{X_0}{s})\cong \\
   \cong \ & \pro{\p{n-1}X_1}{\p{n-1}X_0}{s}
\end{split}
\end{equation*}
that is, $\p{n}X$ satisfies inductive hypothesis b). From this and from hypothesis c) for $X$ we deduce, for all $s\geq 2$ and $1\leq j\leq n-2$
\begin{align*}
    & \p{j,n-2}(\pro{(\p{n}X)_1}{(\p{n}X)_0}{s})=\\
    = \ & \p{j,n-2}(\pro{\p{n-1}X_1}{\p{n-1}X_0}{s})=\\
    = \ & \p{j,n-2}\p{n-1}(\pro{X_1}{X_0}{s})=\\
    = \ & \pro{\p{j,n-1}X_1}{\p{j,n-1}X_0}{s}=\\
    = \ & \pro{\p{j,n-2}(\p{n}X)_1}{\p{j,n-2}(\p{n}X)_0}{s}\;.
\end{align*}
This shows that $\p{n}X$ satisfies inductive hypothesis c). Hence we conclude that $\p{n}X\in\catwg{n-1}$, as required.

\end{proof}
\begin{lemma}\label{lem-flevel-fneq}
    Let $f:X\rw Y$ in $\tawg{n}$ be a levelwise $(n-1)$-equivalence in $\tawg{n-1}$. Then $f$ is an $n$-equivalence.
\end{lemma}
\begin{proof}
By induction on $n$. Let $n=2$. If $f_0$ is an equivalence of categories, $X_0^d\cong Y_0^d$. Hence
\begin{align}\label{eq1-lem-flevel-fneq}
    & Y_1=\uset{a',b'\in Y_0^d}{\coprod} Y(a',b') \cong \uset{fa,fb\in Y_0^d}{\coprod} Y(fa,fb)\;.
\end{align}
\begin{align}\label{eq2-lem-flevel-fneq}
    & X_1=\uset{a,b\in X_0^d}{\coprod}X(a,b)\;.
\end{align}

Since $f_1$ is an equivalence of categories it follows from \eqref{eq1-lem-flevel-fneq} and \eqref{eq2-lem-flevel-fneq} that $f(a,b)$ is an equivalence of categories. Further, $f_k$ is an equivalence of categories for all $k\geq 0$ so that $p f_k=(\p{2}f)_k$ is an isomorphism, hence $\p{2}f$ is an isomorphism; we conclude that $f$ is a 2-equivalence.

Suppose the lemma holds for $(n-1)$ and let $f$ be as in the hypothesis. Since $f_0$ is a $(n-1)$-equivalence in $\cathd{n-1}$, $X_0^d \cong Y_0^d$, so that \eqref{eq1-lem-flevel-fneq} holds. Since $f_1$ is a $(n-1)$-equivalence  it follows from \eqref{eq1-lem-flevel-fneq} and \eqref{eq2-lem-flevel-fneq} that $f(a,b)$ is a $(n-1)$-equivalence for all $a,b\in X_0^d$.

Since $f_k$ is an $(n-1)$-equivalence for all $k\geq 0$, $\p{n-1}f_k=(\p{n}f)_k$ is a $(n-2)$-equivalence. So $\p{n}f$ satisfies the induction hypothesis and is therefore a $(n-1)$-equivalence. In conclusion, $f$ is an $n$-equivalence.
\end{proof}
\begin{remark}\label{rem-local-equiv}
    Applying inductively Lemma \ref{lem-flevel-fneq} it follows immediately that if a morphism $f$ in $\tawg{n}$ is such that $J_n f$ is a levelwise equivalence of categories, then $f$ is an $n$-equivalence.
\end{remark}
\begin{definition}\label{def-local-equiv}
    A morphism $f:X\rw Y$ in $\tawg{n}$ is said to be a local $\nm$-equivalence if for all $a,b \in X_0^d$, $f(a,b):X(a,b)\rw Y(fa,fb)$ is a $\nm$-equivalence in $\tawg{n-1}$.
\end{definition}
\nid In the following proposition, we give a useful description of $n$-equiva-lences in $\tawg{n}$.
\begin{proposition}\label{pro-n-equiv}\

\sk
\begin{itemize}
  \item [a)] Let $f$ be a morphism in $\tawg{n}$ which is an $\nequ$. Then $f$ is a local $(n-1)$-equivalence and $\p{1,n} f$ is an isomorphism.

  \item [b)]  Let $f$ be a morphism  in $\tawg{n}$ which is a local $(n-1)$-equivalence and is such that $\p{1,n} f$ is surjective. Then $f$ is an $\nequ$.

  \item [c)] Let $X\xrw{g} Z \xrw{h} Y$ be morphisms in $\tawg{n}$, $f= hg$ and suppose that $f$ and $h$ are $n$-equivalences. Then $g$ is an $n$-equivalence.

  \item [d)]  Let $X\xrw{g} Z \xrw{h} Y$ be morphisms in $\tawg{n}$, $f= hg$ and suppose that $g$ and $h$ are $n$-equivalences. Then $f$ is an $n$-equivalence.

  \item [e)]  Let $X\xrw{g} Z \xrw{h} Y$ be morphisms in $\tawg{n}$, $f= hg$ and let $g_0^d: X_0^d \rw Z_0^d$ be surjective; suppose that $f$ and $g$ are $n$-equivalences. Then $h$ is an $n$-equivalence.

\end{itemize}
\end{proposition}
\begin{proof}
By induction on $n$. It is clear for $n=1$. Suppose it is true for $n-1$.

a) Let $f:X\rw Y$ be an $\nequ$ in $\tawg{n}$. Then, by definition, $f$ is a local $(n-1)$-equivalence and $\p{n}f$ is a $\equ{n-1}$. Therefore, by induction hypothesis applied to $\p{n}f$, $\p{1,n} f$ is an isomorphism.

\mk
b) Suppose that $f:X\rw Y$ is a local $(n-1)$-equivalence in $\tawg{n}$ and $\p{1,n}f$ is surjective. To show that $f$ is a $\nequ$ we need to show that $\p{n}f$ is a $\equ{n-1}$. For each $a,b\in X_0^d$
\begin{equation*}
    (\p{n}f)(a,b)=\p{n-1}f(a,b)
\end{equation*}
Since $f(a,b)$ is a $(n-1)$-equivalence, $\p{n-1}f(a,b)$ is a $(n-2)$-equivalence; that is, $\p{n}f$ is a local $(n-2)$-equivalence.

Since $\p{1,n} f = \p{1,n-1}\p{n} f$ is surjective, by inductive hypothesis applied to $\p{n}f$ we conclude that $\p{n}f$ is a $\equ{n-1}$ as required.

\mk
c) For all $a,b\in X_0^d$,
\begin{equation}\label{eq1-pro-n-equiv}
    f(a,b)=h(ga,gb)g(a,b)
\end{equation}
with $f(a,b)$ and $h(ga,gb)$ \equ{n-1}s. By inductive hypothesis, $g(a,b)$ is therefore a $\equ{n-1}$.

By hypothesis and by part a), $\p{1,n} f$ and $\p{1,n} h$ are isomorphisms. Since
\begin{equation}\label{eq2-pro-n-equiv}
    \p{1,n} f=(\p{1,n} h)(\p{1,n} g)
\end{equation}
it follows that $\p{1,n} g$ is an isomorphism, hence in particular it is surjective. By part b), this implies that $g$ is a $\equ{n-1}$.
\mk

d) Suppose that $h$ and $g$ are $\nequ$s. By \eqref{eq1-pro-n-equiv}, $f$ is a local $(n-1)$-equivalence and by \eqref{eq2-pro-n-equiv} $\p{1,n} f$ is an isomorphism. By b), $f$ is thus a $\nequ$.
\mk

e) By hypothesis, for each $a',b'\in Z_0^d$, $a'=ga$, $b'=gb$ for $a,b\in X_0^d$. It follows that $h(a',b')=h(ga,gb)$. Since, by induction hypothesis and by \eqref{eq1-pro-n-equiv}, $h(ga,gb)$ is a $\nm$-equivalence, it follows that such is $h(a',b')$. That is, $h$ is a local equivalence.

By hypothesis and by part a), $\p{1,n} f$ and $\p{1,n} g$ are isomorphisms, so by \eqref{eq2-pro-n-equiv}, such is $\p{1,n} h$. We conclude by part b) that $h$ is a $\nequ$.
\end{proof}
\begin{lemma}\label{lem-crit-lev-nequiv}
    Consider the diagram in $\tawg{n}$
    \begin{equation*}
    \xymatrix{
    X \ar^{f}[r]\ar_{\za}[d] & Z \ar@{=}[d] & Y \ar_{g}[l]\ar^{\zb}[d]\\
    X' \ar_{f'}[r] & Z & Y' \ar^{g'}[l]
    }
    \end{equation*}
    with $Z$ discrete. Then:
    \begin{itemize}
      \item [a)] $X\tiund{Z} Y\,,\,X'\tiund{Z'} Y'\in\tawg{n}$, where the pullback are taken in $\funcat{n-1}{\Cat}$.\mk

      \item [b)] $\p{n}(X\tiund{Z}Y)\cong \p{n}X\tiund{\p{n}Z}\p{n}Y$.\mk

      \item [c)] If $\za,\,\zb$ are $n$-equivalences such is
      \begin{equation*}
        (\za,\zb):X\tiund{Z} Y\rw X'\tiund{Z'} Y'\;.
      \end{equation*}
    \end{itemize}
\end{lemma}
\begin{proof}
By induction on $n$. It is clear for $n=1$ since the maps $f,g,f',g'$ are isofibrations as their target is discrete. Suppose, inductively, that the lemma holds for $(n-1)$.

\mk
a) Since pullbacks in $\funcat{n-1}{\Cat}$ are computed pointwise, for each $k\geq 0$
\begin{equation*}
    (X\tiund{Z}Y)_k = X_k \tiund{Z} Y_k \in \funcat{n-2}{\Cat}
\end{equation*}
with $X_k,Y_k\in\tawg{n-1}$. It follows from inductive hypothesis a) that $(X\tiund{Z}Y)_k\in\tawg{n-1}$. Also, $(X\tiund{Z}Y)_0 = X_0 \tiund{Z} Y_0 \in\cathd{n-1}$ since $X_0,Y_0\in\cathd{n-1}$ and $Z$ is discrete, (see \cite[Lemma 3.10 c)]{Pa1}).

To show that $X\tiund{Z}Y\in\tawg{n}$ it remains to prove that the induced Segal maps $\hmu{k}$ for $X\tiund{Z}Y$ are $\nm$-equivalences. We prove this for $k=2$, the case $k>2$ being similar. Note that, by \cite[Lemma 3.10 c)]{Pa1}.
\begin{equation}\label{eq1-crit-lev-nequiv}
    \tens{(X\tiund{Z}Y)_1}{(X\tiund{Z}Y)_0^d} \cong (\tens{X_1}{X_0^d})\tiund{Z}(\tens{Y_1}{Y_0^d})\;.
\end{equation}
Consider the commutative diagram in $\tawg{n-1}$
\begin{equation}\label{eq2-crit-lev-nequiv}
\xymatrix{
X_2 \ar^{}[r]\ar_{\hmu{2}(X)}[d] & Z \ar@{=}[d] & Y_2 \ar^{}[l] \ar^{\hmu{2}(Y)}[d]\\
\tens{X_1}{X_0^d} \ar^{}[r] & Z & \tens{Y_1}{Y_0^d} \ar^{}[l]
}
\end{equation}
The vertical maps are the induced Segal maps for $X$ and $Y$, hence they are $\nm$-equivalences. By inductive hypothesis b) applied to \eqref{eq2-crit-lev-nequiv} we conclude that the induced map of pullbacks is a $\nm$-equivalence. By \eqref{eq1-crit-lev-nequiv} the latter is the induced Segal map $\hmu{2}$ for $X\tiund{Y}Z$. The proof for $k>2$ is similar and we conclude that $X\tiund{Y}Z\in\tawg{n}$.
\mk

b) By \eqref{eq-wg-ps-cat}, for all $\uk\in\dop{n-2}$
\begin{equation}\label{eq1-lem-crit-lev-nequiv}
    (\p{n}(X\tiund{Z}Y))_{\uk} = p (X\tiund{Z}Y)_{\uk}\;.
\end{equation}
Since pullbacks in $\funcat{n-1}{\Cat}$ are computed pointwise and $p$ commutes with pullbacks over discrete objects, we have
\begin{equation}\label{eq2-lem-crit-lev-nequiv}
\begin{split}
    & p (X\tiund{Z}Y)_{\uk} = p (X_{\uk}\tiund{Z_{\uk}}Y_{\uk})_{\uk} = pX_{\uk}\tiund{pZ_{\uk}}pY_{\uk}=\\
   = & (\p{n} X)_{\uk}\tiund{(\p{n}Z)_{\uk}} (\p{n} Y)_{\uk}\;.
\end{split}
\end{equation}
Since this holds for all $\uk$, \eqref{eq1-lem-crit-lev-nequiv} and \eqref{eq2-lem-crit-lev-nequiv} imply b).

\mk
c) For each $(a,b),\, (c,d)\in (X\tiund{Z}Y)^d_0 = X_0^d\tiund{Z}Y_0^d$ we have %
\begin{equation*}
\begin{split}
    & (X\tiund{Z}Y)((a,b),(c,d))=X(a,c)\times Y(b,d) \\
    & (X'\tiund{Z}Y')((fa,fb),(gc,gd))=X'(fa,gc)\times Y'(fb,gd)\;.
\end{split}
\end{equation*}
Since $\za,\zb$ are $n$-equivalences, $\za(a,c)$ and $\zb(b,d)$ are $\nm$-equivalences, hence such is
\begin{equation*}
    (\za,\zb)((a,b),(c,d))=\za(a,c)\times \zb(b,d)\;.
\end{equation*}
Since $\p{n}$ commutes with pullbacks over discrete objects for each $n$, so does $\p{1,n}$, hence
\begin{equation*}
    \p{1,n}(X\tiund{Z}Y) = \p{1,n}X\tiund{Z}\p{1,n}Y\;.
\end{equation*}
From the hypothesis and from Proposition \ref{pro-n-equiv} a), $\p{1,n}\za$ and $\p{1,n}\zb$ are isomorphisms, thus so is $\p{1,n}(\za,\zb)$. Since, from above, $(\za,\zb)$ is a local $\nm$-equivalence, we conclude by Proposition \ref{pro-n-equiv} b) that $(\za,\zb)$ is a $n$-equivalence.
\end{proof}

\begin{proposition}\label{pro-crit-lev-nequiv}
    Let $f:X\rw Y$ be a morphism in $\tawg{n}$ with $n\geq 2$, such that
    \begin{itemize}
      \item [a)] $f$ is a $\nequ$.
      \item [b)] $\p{n-1}X_0 \cong \p{n-1}Y_0$,
      \item [c)] For each $1\leq r < n-1$ and all $k_1, \ldots, k_r \geq 0$,
      \begin{equation*}
        \p{n-r-1}X_{k_1, \ldots, k_r,\, 0}\cong \p{n-r-1}Y_{k_1, \ldots, k_r,\, 0}\;.
      \end{equation*}
      Then $J_nf$ is a levelwise equivalence of categories.
    \end{itemize}
\end{proposition}
\begin{proof}
By induction on $n$. Let $f:X\rw Y$ be a 2-equivalence in $\tawg{2}$ such that $p X_0=X_0^d\cong p Y_0=Y_0^d$. Since $f_0$ is a morphism in $\cathd{}$ it follows that $f_0$ is an equivalence of categories. Since $f(a,b)$ is an equivalence of categories for all $a,b\in X_0^d$ there is an equivalence of categories
\begin{equation*}
    f_1:X_1=\underset{a,b\in X_0^d}{\cop}X(a,b)\rw Y_1= \underset{a',b'\in Y_0^d}{\cop}Y(a',b')=\underset{fa,fb\in Y_0^d}{\cop}Y(fa,fb).
\end{equation*}

Hence there are equivalences of categories for $k\geq 2$
\begin{equation*}
    X_k \sim \pro{X_1}{X_0^d}{k}\sim \pro{Y_1}{Y_0^d}{k}\sim Y_k\;.
\end{equation*}
In conclusion $X_k\sim Y_k$ for all $k\geq 0$.

Suppose, inductively, that the statement holds for $(n-1)$ and let $f:X\rw Y$ be as in the hypothesis. We show that $J_{n-1}f_k$ is a levelwise equivalence of categories for each $k\geq 0$ by showing that $f_k$ satisfies the inductive hypothesis. It then follows that $J_n f$ is a levelwise equivalence of categories since
\begin{equation*}
    (J_n f)_{k_1...k_{n-1}}=(J_{n-1}f_{k_1})_{k_2...k_{n-1}}\;.
\end{equation*}

Since  $X_0\in\cathd{n-1}$, from b) and \cite[Lemma 3.8]{Pa1} we obtain
\begin{equation}\label{eq1-pro-crit-lev-nequiv}
    X_0^d=\p{1,n}X_0 \cong \p{1,n}Y_0 \cong Y_0^d\;.
\end{equation}
Thus, by \cite[Lemma 3.8]{Pa1} again, $f_0:X_0\rw Y_0$ is a $\equ{n-1}$. Further, by hypothesis c),
\begin{align*}
    & \p{n-2}X_{00} \cong \p{n-2}Y_{00}\\
    & \p{n-r-2}X_{0\,k_1...k_r\,0} \cong \p{n-r-2}Y_{0\,k_1...k_r\,0}
\end{align*}
for each $1\leq r < n-2$ and all $k_1...k_r$. Thus $f_0:X_0\rw Y_0$ satisfies the inductive hypothesis and we conclude that $f_0$ is a levelwise equivalence of categories. By \eqref{eq1-pro-crit-lev-nequiv} we also have
\begin{equation*}
    f_1:\underset{a,b\in X_0^d}{\cop}X(a,b)\rw Y_1 = \underset{a',b'\in Y_0^d}{\cop}Y(a',b')=\underset{fa,fb\in Y_0^d}{\cop}Y(fa,fb)\;.
\end{equation*}
Since $f$ is a local $\equ{n-1}$, it follows that $f_1:X_1 \rw Y_1$ is a $\equ{n-1}$. Further, by hypothesis c)
\begin{align*}
    & \p{n-2}X_{10} \cong \p{n-2}Y_{10}\\
    & \p{n-r-2}X_{1\,k_1...k_r\,0} \cong \p{n-r-2}Y_{1\,k_1...k_r\,0}
\end{align*}
for all $1 \leq r < n-2$. Thus $f_1$ satisfies the inductive hypothesis, and is therefore a levelwise equivalence of categories.

For each $k\geq 2$ consider the map
\begin{equation*}
    (f_1,...,f_1):\pro{X_1}{X_0^d}{k}\rw \pro{Y_1}{Y_0^d}{k}\;.
\end{equation*}
Since $X_0^d \cong Y_0^d$ and, from above, $f_1$ is a $\equ{n-1}$, then $(f_1,...,f_1)$ is also a $\equ{n-1}$.

There is a commutative diagram in $\tawg{n-1}$
\begin{equation*}
\xymatrix{
X_k \ar^(0.3){\hmuk}[rr] \ar_{f_k}[d] && \pro{X_1}{X_0^d}{k} \ar^{(f_1,...,f_1)}[d]\\
Y_k \ar_(0.3){\hmuk}[rr] && \pro{Y_1}{Y_0^d}{k}
}
\end{equation*}
where the horizontal induced Segal maps are $(n-1)$-equivalences since $X,Y\in\tawg{n-1}$ and the right vertical map is a $(n-1)$-equivalence from above. It follows from Proposition \ref{pro-n-equiv} c) and d) that $f_k$ is a $(n-1)$-equivalence. Further, from hypothesis c),
\begin{align*}
    & \p{n-1}X_{k0} \cong \p{n-2}Y_{k0}\\
    & \p{n-r-2}X_{k\,k_1...k_r\,0} \cong \p{n-r-2}Y_{k\,k_1...k_r\,0}\;.
\end{align*}
Thus $f_k$ satisfies the induction hypothesis and we conclude that $f_k$ is a levelwise equivalence of categories.

In conclusion, $f_k$ is a levelwise equivalence of categories for all $k\geq 0$. Since this holds for each $k\geq0$ this implies that $f$ is a levelwise equivalence of categories.
\end{proof}

\section{The functor $\q{n}$.}\label{sec-func-qn}
This section introduces the functor
\begin{equation*}
  \q{n}:\tawg{n}\rw \tawg{n-1}
\end{equation*}

This functor is a higher dimensional generalization of the connected component functor $q:\Cat\rw\Set$ and comes equipped with a morphism
 \begin{equation*}
   \zg\up{n}:X\rw \di{n}\q{n}X
 \end{equation*}
  for each $X\in \tawg{n}$, where $\di{n}$ is as in Notation \ref{not-ner-func-dirk} . It will be used crucially in Section \ref{sec-cat-lta} to replace a weakly globular \nfol category $X$ with a simpler one (Theorem \ref{the-repl-obj-1}). This will involve taking pullbacks along the map $\zg\up{n}$. The last part of this Section establishes several properties of these pullbacks needed in Section \ref{sec-cat-lta}.
\begin{proposition}\label{pro-post-trunc-fun}
    There is a functor $\qn:\tawg{n}\rw \tawg{n-1}$ making the following diagram commute.
    \begin{equation}\label{eq-lem-post-trunc}
    \xymatrix@C=30pt{
    \tawg{n} \ar^{J_{n}}[rr]\ar_{\qn}[d] && \funcat{n}{\Cat} \ar^{\ovl{q}}[d]\\
    \tawg{n-1} \ar_{\Nb{n-1}}[rr] & & \funcat{n-1}{\Set}
    }
    \end{equation}
    where $q:\Cat\rw \Set$ is the connected component functor. The functor $\qn$ sends $n$-equivalences to $\nm$-equivalences and preserves pullbacks over discrete objects. If $X\in\cathd{n}$, then $\q{n}X=\p{n}X$; further, for each $X\in \tawg{n}$, there is a map $\zgu{n}:X\rw \dn\qn X$ natural in $X$.
\end{proposition}
\begin{proof}
By induction on $n$; for $n=1$, $\q{1}=q:\Cat\rw\Set$ is the connected components functor which, by Lemma \ref{lem-q-pres-fib-pro}, has the desired properties. If $X\in\cathd{}$, in particular $X$ is a groupoid, so $pX=qX$.

Suppose there is a $\q{n-1}$ with the desired properties, let $X\in\tawg{n}$. We claim that $ \ovl{\q{n-1}}X\in\tawg{n-1}$. In fact, for each $s\geq 0$, by induction hypothesis
\begin{equation*}
    ( \ovl{\q{n-1}}X)_s=\q{n-1}X_s\in\tawg{n-2}\;.
\end{equation*}
Also, by induction hypothesis and by Definition \ref{def-hom-dis-ncat},
\begin{equation*}
    ( \ovl{\q{n-1}}X)_0 =\q{n-1}X_0 = \p{n-1}X_0 \in \cathd{n-1}
\end{equation*}
as $X_0 \in \cathd{n-1}$. Further, since
\begin{equation*}
    \hmuk : X_k\rw \pro{X_1}{X_0^d}{k}
\end{equation*}
is a $\equ{n-1}$, by induction hypothesis the map
\begin{equation*}
    \q{n-1}X_k\rw\q{n-1}(\pro{X_1}{X_0^d}{k})\cong \pro{\q{n-1}X_1}{X_0^d}{k}
\end{equation*}
is a $\equ{n-2}$; in this we used the fact that $\q{n-1}X_0^d\cong X_0^d$, which follows from diagram \eqref{eq-lem-post-trunc} at step $(n-1)$.

This shows that $ \ovl{\q{n-1}}X\in\tawg{n-1}$. We therefore define
\begin{equation*}
    \q{n}X = \ovl{\q{n-1}}X\;.
\end{equation*}
The fact that $\q{n}$ satisfies diagram \eqref{eq-lem-post-trunc} is immediate from the definitions and the induction hypothesis.

If $X\in\cathd{n}$, by definition $X_k\in\cathd{n-1}$ for each $k$, so by induction hypothesis $\p{n-1}X_k=\q{n-1}X_k$. It follows that $(\p{n}X)_k=(\q{n}X)_k$ for all $k$. That is $\p{n}X=\q{n}X$.

Let $f:X\rw Y$ be a $\nequ$ in $\tawg{n}$ and let $a,b\in X_0^d$. Then from the definitions
\begin{equation*}
    (\q{n}f)(a,b)=\q{n-1}f(a,b)\;.
\end{equation*}
Since $f(a,b)$ is a $\equ{n\mi1}$, by induction hypothesis $\q{n-1} f(a,b)$ is a $\equ{n-2}$. By Proposition \ref{pro-n-equiv}, to prove that $\q{n}f$ is a $\equ{n-1}$, it is enough to show that $\p{1,n-1}\q{n}f$ is an surjective.

Recall that for any category $\clC$ there is a surjective map $p\,\clC\rw q\,\clC$ natural in $\clC$. Applying this levelwise to $J_nX$ we obtain a map
\begin{equation*}
    \za_X\up{n}:\p{n}X \rw \q{n}X
\end{equation*}
natural in $X$. The map $\za\up{n}$ induces  a functor
\begin{equation}\label{eq-post-trunc-fun}
 \p{2,n-1}\za_X\up{n}: \p{2,n-1}\p{n}X \rw \p{2,n-1}\q{n}X\;,
\end{equation}
which is identity on objects. In fact, on object this map is given by
\begin{equation*}
   \p{1,n-2}\p{n-1}X_0 \rw \p{1,n-2}\q{n-1}X_0
\end{equation*}
and since $X_0\in\cathd{n-1}$, $\p{n-1}X_0=\q{n-1}X_0$ so this map is the identity. It follows that the map in $\Set$
\begin{equation*}
    \p{1,n-1}\za_X\up{n-1}: \p{1,n-1}\p{n}X \rw \p{1,n-1}\q{n}X
\end{equation*}
is surjective. We thus have a commuting diagram
\begin{equation*}
\xymatrix{
\p{1,n-1}\p{n}X \ar^{\p{1,n-1}\p{n}f}[rr] \ar_{\p{1,n-1}\za_X\up{n}}[d] && \p{1,n-1}\p{n}Y \ar^{\p{1,n-1}\za_Y\up{n}}[d]\\
\p{1,n-1}\q{n}X \ar^{\p{1,n-1}\q{n}f}[rr] && \p{1,n-1}\q{n}Y
}
\end{equation*}
in which the top arrow is an isomorphism (by Proposition \ref{pro-n-equiv}) and from above the vertical arrows are surjective. It follows that the bottom map is also surjective. By Proposition \ref{pro-n-equiv} b) we conclude that $\q{n}f$ is a $\equ{n-1}$.

Finally, the map $\zgu{n}:X\rw \q{n}X$ is given levelwise by the maps $X_s\rw \q{n-1}X_s$, which exist by induction hypothesis.
\end{proof}
\begin{remark}\label{rem-spec-isofib}
    For each $X\in \tawg{n}$, from the proof of Proposition \ref{pro-post-trunc-fun} the functor $\p{2,n-1}\za_X\up{n}$ is identity on objects. It is also surjective on morphisms since, by the proof of Proposition \ref{pro-post-trunc-fun}, the map
    \begin{equation*}
        (\p{2,n-1}\za_X\up{n})_1:\p{1,n-1}\p{n-1} X_1 \rw \p{1,n-2}\q{n-1} X_1 \;,
    \end{equation*}
    is surjective. It follows that $\p{2,n-1}\za_X\up{n}$ is an isofibration.
\end{remark}
\begin{lemma}\label{lem-spec-pulbk-eqr}
    Let $X\in\cathd{n}$, $Z\in\cathd{n-1}$, $r: Z\rw \qn X$. Consider the pullback in $\funcat{n-1}{\Cat}$
    \begin{equation*}
        \xymatrix@R=35pt @C=40pt{
        P \ar^{}[r] \ar^{}[d] & X \ar^{\zgu{n}}[d] \\
        \dn Z \ar_{\dn r}[r] & \dn \qn X
        }
    \end{equation*}
    then $P\in\cathd{n}$ and $\p{n}P=Z$.
\end{lemma}
\begin{proof}
By induction on $n$. For $n=1$, since $\di{1}\q{1}X$ is discrete, the map $\zgu{1}:X\rw \di{1}\q{1}X=X^d$ is an isofibration. Therefore, since $\zgu{1}$ is an equivalence of categories (as $X\in\cathd{}$) we have an equivalence of categories
\begin{equation*}
    P=\di{1}Z\tiund{\di{1}\q{1}X} X\simeq \di{1}Z\tiund{\di{1}\q{1}X} \di{1}\q{1}X = \di{1}Z
\end{equation*}
Thus $P\in\cathd{}$ and $pP=Z$.

Suppose, inductively, that the lemma holds for $n-1$ and let $P$ be as in the hypothesis. Since pullbacks in $\funcat{n-1}{\Cat}$ are computed pointwise, for each $k\geq 0$ we have a pullback in $\funcat{n-2}{\Cat}$
\begin{equation*}
\xymatrix{
P_k \ar^{}[rr] \ar^{}[d] && X_k \ar^{}[d]\\
\di{n-1}Z_k \ar^{}[rr] && \di{n-1}\q{n-1}X_k
}
\end{equation*}
where $X_k\in\cathd{n-1}$ (since $X\in \cathd{n}$) and $Z_k \in \cathd{n-2}$ (since $Z\in\cathd{n-1}$). By induction hypothesis, we conclude that $P_k\in\cathd{n-1}$.

We now show that, for each $k \geq 2$
\begin{equation}\label{eq1-lem-spec-pulbk-eqr}
    P_k\cong \pro{P_1}{P_0}{k}\;.
\end{equation}
We illustrate this for $k=2$, the case $k>2$ being similar. Since $X\in\cathd{n}$, $\q{n}X = \p{n}X\in\cathd{n-1}$, so
\begin{equation*}
\begin{split}
    & \q{n-1}X_2=\p{n-1}X_2 =\p{n-1}(\tens{X_1}{X_0})= \\
   =\ & \tens{\p{n-1} X_1}{\p{n-1} X_0}=\tens{\q{n-1} X_1}{\q{n-1} X_0}\;.
\end{split}
\end{equation*}
Since $X_2 \cong \tens{X_1}{X_0}$ and $Z_2 \cong \tens{Z_1}{Z_0}$, it follows from Lemma \ref{lem-char-obj-II} that $P_2 \cong \tens{P_1}{P_0}$.

To prove that $P\in\cathd{n}$ it remains to show that $\p{n}P\in\cathd{n-1}$. Since $p$ commutes with fiber products over discrete objects, for each $\us\in\dop{n-1}$ we have
\begin{align*}
    &(\p{n}P)_{\us}=p\,P_{\us}=
    p(d\,Z_{\us}\tiund{d q X_{\us}} X_{\us})
     = Z_{\us}\tiund{qX_{\us}}p X_{\us}=Z_{\us}
\end{align*}
where we used the fact that, since $X_{\us}$ is a groupoid, $pX_{\us}=qX_{\us}$. Since this holds for each ${\us}$ we conclude that $\p{n}P=Z\in\cathd{n-1}$ as required.
\end{proof}
\begin{lemma}\label{lem-p2-n-1}
    Let $Y\in\tawg{n}$ and let
    \begin{equation}\label{eq1-lem-p2-n-1}
        X\rw \q{n}Y \lw \p{n}Y
    \end{equation}
    be a diagram in $\tawg{n-1}$ such that $X\tiund{\q{n}Y}\p{n}Y\in \tawg{n-1}$. Then for all $1\leq j \leq n-1$
      \begin{equation*}
        \p{j,n-1}(X\tiund{\q{n}Y}\p{n}Y)=\p{j,n-1}X\tiund{\p{j,n-1}\q{n}Y}\p{j,n-1}\p{n}Y\;.
      \end{equation*}
\end{lemma}
\begin{proof}\

By induction on $n$. For $n=2$, the functor $\p{2}Y\rw \q{2}Y$ is the identity on objects, therefore by Lemma \ref{lem-iso-cla-obj-fib-pro}
\begin{equation*}
    p(X\tiund{\q{2}Y}\p{2}Y)=pX\tiund{p\q{2}Y}p\p{2}Y\;.
\end{equation*}
Suppose, inductively, that it holds for $n-1$. Then for each $k\geq 0$
\begin{equation}\label{eq2-lem-p2-n-1}
\begin{split}
   & (\p{j,n-1}(X\tiund{\q{n}Y}\p{n}Y))_k= \\
   =\; &\p{j-1,n-2}(X_k\tiund{\q{n-1}Y_k}\p{n-1}Y_k)=\\
   =\; &\p{j-1,n-2}X_k \tiund{\p{j-1,n-2}\q{n-1}Y_k} \p{j-1,n-2}\p{n-1}Y_k=\\
   =\; &(\p{j,n-2}X)_k \tiund{(\p{j,n-2}\q{n-1}Y)_k} (\p{j,n-2}\p{n-1}Y)_k\;.
\end{split}
\end{equation}
Since this holds for each $k\geq 0$, the lemma  follows.
\end{proof}
\begin{proposition}\label{pro-spec-plbk-pscatwg}
    Let $\di{n}A\xrw{\ \di{n}f\ }\di{n}\q{n}C \xlw{\ g\ } C$ be a diagram in $\tawg{n}$ where $f: A\rw\q{n}C$ is a morphism in $\tawg{n-1}$ and consider the pullback in $\funcat{n-1}{\Cat}$
    \begin{equation*}
        P=\di{n}A \tiund{\di{n}\q{n}C} C\;.
    \end{equation*}
    \begin{itemize}
    \item [a)] Then $P\in \tawg{n}$ and
    \begin{equation}\label{eq-pro-spec-plbk-pscatwg0}
        \p{1,n}P = \p{1,n-1} A \tiund{\p{1,n-1}\q{n}C} \p{1,n} C\;.
    \end{equation}
    \item [b)]

    Consider the commutative diagram in $\tawg{n}$
    \begin{equation}\label{eq-pro-spec-plbk-pscatwg}
        \xymatrix@R=35pt @C=60pt{
        \di{n}A \ar^{\di{n}f}[r] \ar_{a}[d] & \di{n}\q{n} C \ar^{b}[d] & C \ar^{c}[d] \ar_{g}[l] \\
        \di{n}D \ar_{\di{n}h}[r] & \di{n}\q{n} F  & F \ar^{l}[l]
        }
    \end{equation}
    where $a,b,c$ are $\nequ$s. Then the induced maps of pullbacks
    \begin{equation*}
    (a,c): \di{n}A\tiund{\di{n}\q{n} C} C \rw \di{n}D \tiund{\di{n}\q{n} F} F
    \end{equation*}
    is a $n$-equivalence in $\tawg{n}$.

    \medskip
    \item [c)] If $f$ is an $(n-1)$-equivalence, $P\xrw{w}C$ is an $n$-equivalence.

    \end{itemize}
\end{proposition}
\begin{proof}
By induction on $n$. When $n=1$, the maps $f,g,h,l$ are isofibrations since their target is a discrete category. Therefore, since $a,b,c$ are equivalences of categories, the induced map of pullbacks
\begin{equation*}
    (a,c):d A\tiund{d q C} C \rw d D \tiund{d q F} F
\end{equation*}
is an equivalence of categories and $p (d A\tiund{dqC}C) = A \tiund{qC} pC$.

Suppose, inductively, that the proposition holds for $n-1$.

\bk
a) We have
\begin{equation*}
    (\di{n}A\tiund{\di{n}\q{n} C} C)_0=\di{n-1}A_0\tiund{\di{n-1}\q{n-1} C_0} C_0\;.
\end{equation*}
Since, by hypothesis, $A \in\tawg{n-1}$ and $C$ is in $\tawg{n}$, by definition $A_0 \in\cathd{n-2}$ and $C_0 \in \cathd{n-1}$. Therefore by Lemma \ref{lem-spec-pulbk-eqr}
\begin{equation*}
    \di{n-1}A_0\tiund{\di{n-1}\q{n-1} C_0} C_0\in \cathd{n-1}\;.
\end{equation*}
Further, for each $k\geq 1$,
\begin{equation*}
    (\di{n}A\tiund{\di{n}\q{n} C} C)_k = \di{n-1}A_k\tiund{\di{n-1}\q{n-1} C_k} C_k
\end{equation*}
where, by hypothesis, $A_k \in\tawg{n-2}$ and $C_k \in \tawg{n-1}$. It follows by inductive hypothesis that
\begin{equation*}
    \di{n-1}A_k\tiund{\di{n-1}\q{n-1} C_k} C_k \in \tawg{n-1}\;.
\end{equation*}
To prove that $\di{n}A\tiund{\di{n}\q{n} C} C \in \tawg{n}$, it remains to show that its induced Segal maps $\hmu{s}$ are $\equ{n-1}$s for all $s\geq 2$. We show this for $s=2$, the case $s>2$ being similar. We have
\begin{equation*}
\begin{split}
    & \hmu{2}:(\di{n}A\tiund{\di{n}\q{n} C} C)_2 \rw \\
    & \rw \tens{(\di{n}A\tiund{\di{n}\q{n} C} C)_1}{(\di{n}A\tiund{\di{n}\q{n} C} C)_0^d}\;.
\end{split}
\end{equation*}
By Lemma \ref{lem-spec-pulbk-eqr},
\begin{equation*}
\begin{split}
    & (\di{n}A\tiund{\di{n}\q{n} C} C)_0^d= (\p{n-1}(\di{n-1}A_0\tiund{\di{n-1}\q{n-1} C_0} C_0))^d=A_0^d = \\
    & = A_0^d\tiund{\di{n-1}\q{n-1} C_0^d} C_0^d
\end{split}
\end{equation*}
where we used the fact that, since $\q{n-1}\di{n-1}=\Id$,
\begin{equation*}
\begin{split}
    & \di{n-1}\q{n-1}C_0^d= \di{n-1}\q{n-1}\di{n-1}...\di{1}p...\p{n-1}C_0= \\
   =\; &\di{n-1}\di{n-2}...\di{1}p...\p{n-1}C_0=C_0^d\;.
\end{split}
\end{equation*}
Recalling that $\q{n-1}$ commutes with pullbacks over discrete objects, we obtain
\begin{equation*}
\begin{split}
    & \tens{(\di{n}A\tiund{\di{n}\q{n}C}C)_1}{(\di{n}A\tiund{\di{n}\q{n}C}C)_0^d}= \\
    & \resizebox{1.0\hsize}{!}{$ =\tens{(\di{n-1}A_1\tiund{\di{n-1}\q{n-1}C_1}C_1)}{(A_0^d\tiund{\di{n-1}\q{n-1}C_0^d}C_0^d)}=$} \\
   & = \di{n-1}(\tens{A_1}{A_0^d})\tiund{\di{n-1}\q{n-1}(\tens{C_1}{C_0^d})}(\tens{C_1}{C_0^d})\;.
\end{split}
\end{equation*}
On the other hand
\begin{equation*}
    (\di{n}A\tiund{\di{n}\q{n}C} C)_2 = \di{n-1}A_2 \tiund{\di{n-1}\q{n-1}C_2} C_2 \;.
\end{equation*}
Hence we see that the map $\hmu{2}$ is the induced map on pullbacks from the diagram in $\tawg{n-1}$
\begin{equation}\label{eq1-pro-spec-plbk-pscatwg}
\xymatrix@C=15pt{
\di{n-1}A_2 \ar^{}[rr] \ar^{}[d] && \di{n-1}\q{n-1}C_2 \ar^{}[d] && C_2 \ar^{}[ll] \ar^{}[d]\\
\di{n-1}(\tens{A_1}{A_0^d}) \ar^{}[rr] && \di{n-1}\q{n-1}(\tens{C_1}{C_0^d}) && \tens{C_1}{C_0^d} \ar^{}[ll]
}
\end{equation}
In this diagram, the left and right vertical maps are $\nm$-equivalences since they are induced Segal maps of $\di{n}A$ and $C$ respectively. The map
\begin{equation*}
    \q{n-1}C_2 \rw \q{n-1}(\tens{C_1}{C_0^d}) = \tens{\q{n-1}C_1}{(\q{n-1}C_0)^d}
\end{equation*}
is the induced Segal map for $\q{n}C\in\tawg{n-1}$, and is therefore a $(n-2)$-equivalence. The central vertical map in \eqref{eq1-pro-spec-plbk-pscatwg} is therefore in particular a $(n-1)$-equivalence.  We can therefore apply the induction hypothesis b) to the diagram \eqref{eq1-pro-spec-plbk-pscatwg} and conclude that the induced map on pullbacks is a $\equ{n-1}$. That is $\hmu{2}$ is a $\nm$-equivalence.

Similarly one shows that $\hmu{s}$ is a $\nm$-equivalence for all $s \geq 2$. This concludes the proof that $\di{n}A \tiund{\di{n}\q{n}C} C \in \tawg{n}$ and \eqref{eq-pro-spec-plbk-pscatwg0} follows from Lemma \ref{lem-p2-n-1}.

\bk
b) We now show that $(a,c)$ is a $\nequ$. By Proposition \ref{pro-n-equiv} b), it is enough to show that it is a local $\equ{n-1}$ and that $\p{1,n}(a,c)$ is an isomorphism.

Let
\begin{equation*}
    (x_1,y_1),(x_2,y_2) \in (\di{n}A \tiund{\di{n}\q{n}C} C)_0^d = A_0^d \tiund{\di{n-1}\q{n-1}C_0^d} C_0^d.
\end{equation*}
Then the map
\begin{equation*}
    (a,c)((x_1,y_1),(x_2,y_2))
\end{equation*}
is the induced map on pullbacks of the diagrams
\begin{equation}\label{eq2-pro-spec-plbk-pscatwg}
\xymatrix@C=14pt{
\di{n}A(x_1,x_2) \ar^(0.38){f(x_1,x_2)}[rr] \ar_{a(x_1,x_2)}[d] && \di{n-1}\q{n-1}C(fx_1,fx_2) \ar^{b(fx_1,fx_2)}[d] && C(y_1,y_2) \ar_(0.38){g(y_1,y_2)}[ll] \ar_{c(y_1,y_2)}[d]\\
\di{n-1}D(a x_1,a x_2) \ar_(0.45){h(a x_1,a x_2)}[rr] && \di{n-1}\q{n-1}F(h a x_1,h a x_2) && C(cy_1,cy_2) \ar^(0.35){C(cy_1,cy_2)}[ll]
}
\end{equation}
The vertical maps in \eqref{eq2-pro-spec-plbk-pscatwg} are $\equ{n-1}$s. By the inductive hypothesis b), since $a,b,c$ are $n$-equivalences we conclude that the induced map of pullbacks is a $\equ{n-1}$. This shows that $(a,c)$ is a local $\equ{n-1}$. By Lemma \ref{lem-p2-n-1}, $\p{n-1}(a,c)=(\p{1,n}a,\p{1,n}c)$.

Applying the functor $\p{1,n}$ to the diagram  \eqref{eq-pro-spec-plbk-pscatwg} we obtain a commutative diagram in $\Set$
\begin{equation*}
\xymatrix{
\p{1,n-1}A \ar^{}[rr] \ar_{\p{1,n-1}a}[d] && \p{1,n-1}\q{n}C \ar^{\p{1,n}b}[d] && \p{1,n-1}\p{n}C \ar^{}[ll] \ar^{\p{1,n}c}[d]\\
\p{1,n-1}D \ar^{}[rr] && \p{1,n-1}\q{n}F && \p{1,n-1}\p{n}C \ar^{}[ll]
}
\end{equation*}
Since, by hypothesis, $a,b$ and $c$ are $\nequ$s, by Proposition \ref{pro-n-equiv} a) the vertical maps are isomorphisms, hence such is $\p{1,n}(a,c)$ as required. We conclude from Proposition \ref{pro-n-equiv} b) that $(a,c)$ is an $n$-equivalence.

\bk
c) By a), we have a commutative diagram in $\tawg{n}$
\begin{equation*}
    \xymatrix{
    \di{n}A \ar^{f}[rr] \ar_{f}[d] && \di{n}\q{n}C \ar@{=}[d] && C \ar_{g}[ll] \ar@{=}[d] \\
    \di{n}\q{n}C \ar@{=}[rr] && \di{n}\q{n}C  && C \ar^{g}[ll]
    }
\end{equation*}
in which the vertical maps are $n$-equivalences. It follows by b) that the induced map of pullbacks
\begin{equation*}
    P= \di{n}A\tiund{\di{n}\q{n}}C \xrw{w} \di{n}\q{n} C \tiund{\di{n}\q{n}} C = C
\end{equation*}
is an $n$-equivalence.

\end{proof}

\section{The category $\mathbf{\lta{n}}$.}\label{sec-cat-lta}
In this section we introduce the subcategory $\lta{n}$ of $\tawg{n}$. We will show in Section \ref{sec-wg-tam-to-psefun} how to associate to this category a Segalic pseudo-functor, which in turn will lead to the contruction of the rigidification functor $Q_n$.

 The main result of this section is  Theorem \ref{the-repl-obj-1} establishing that that every object of $\tawg{n}$ can be approximated up to an $n$-equivalence with an object of $\tawg{n}$.
\begin{definition}\label{def-ind-sub-ltawg}
    Define inductively the subcategory $\lta{n}\subset\tawg{n}$. For $n=2$, $\lta{2}=\tawg{2}$. Suppose we defined $\lta{n-1}\subset \tawg{n}$. Let $\lta{n}$ be the full subcategory of $\tawg{n}$ whose objects $X$ are such that
    \begin{itemize}
      \item [i)] $X_k\in\lta{n-1}$ for all $k\geq 0$.\bk

      \item [ii)] The maps in $\funcat{n-2}{\Cat}$
    \begin{equation*}
        v_k: J_{n-1} X_k \rw J_{n-1}(\pro{X_1}{\di{n-1}\p{n-1}X_0}{k})
    \end{equation*}
    are levelwise equivalences of categories for all $k\geq 2$\bk

     \item [iii)] $\p{n}X\in\catwg{n-1}$.
    \end{itemize}
\end{definition}
\begin{definition}\label{def-ltawg-equiv}
    Let $\lnta{n}{n}$ be the full subcategory of $\tawg{n}$ whose objects $X$ are such that
    \begin{itemize}
      \item [i)] The maps in $\funcat{n-2}{\Cat}$
      \begin{equation*}
        v_k: J_{n-1}X_k \rw J_{n-1}(\pro{X_1}{\di{n-1}\p{n-1}X_0}{k})
      \end{equation*}
      are levelwise equivalences of categories for all $k\geq 2$\bk

      \item [ii)] $\p{n}X\in\catwg{n-1}$.
    \end{itemize}
\end{definition}
\begin{lemma}\label{lem-lev-wg-pscat} 
    Let $X\in\tawg{n}$. Then $X\in \lta{n}$ if and only if
    \begin{itemize}
      \item [a)] $X\in \lnta{n}{n}$.\bk

      \item [b)] For each $1 < r \leq n-1$ and each $k_1,...,k_{n-r}\in\Dop$, $X_{k_1,...,k_{n-r}}\in \lnta{r}{r}$
    \end{itemize}
\end{lemma}
\begin{proof}
By induction on $n$. It holds for $n=2$ since $\lta{2}=\Lb{2}\tawg{2}=\tawg{2}$. Suppose it holds for $(n-1)$ and let $X\in\lta{n}$. Then by definition $X\in\Lb{n}\tawg{n}$. Also by definition $X_{k_1\cdots k_{n-r}}\in\lta{r}$ and therefore $X_{k_1\cdots k_{n-r}}\in\Lb{r}\tawg{r}$.

Conversely, suppose that $X\in\tawg{n}$ satisfies a) and b). By a), $v_k$ is a levelwise equivalence of categories and $\p{n}X\in\catwg{n-1}$; by b), $X_k\in\Lb{n-1}\tawg{n-1}$ and, further, $X_k$ itself satisfies b). Thus by induction hypothesis applied to $X_k$ we conclude that $X_k\in\lta{n-1}$. By definition, this shows that $X\in\lta{n}$.
\end{proof}

\begin{corollary}\label{cor-lev-wg-ncat}
    Let $X\in\catwg{n}$. Then $X\in \lta{n}$.
\end{corollary}
\begin{proof}
We claim that $X\in\lnta{n}{n}$. We prove the claim by induction on $n$. For $n=2$, if $X\in\catwg{2}$, for each $s\geq 2$ there is an equivalence of categories
\begin{equation*}
    X_s\simeq \pro{X_1}{X_0^d}{s}=\pro{X_1}{\di{1}\p{1}X_0}{s}\;.
\end{equation*}
Thus, by definition, $X\in\lnta{2}{2}$. Suppose, inductively, that the statement holds for $n-1$ and let $X\in\catwg{n}$. Then, by Proposition \ref{pro-property-WG-nfol}, $X_k\up{2}\in \catwg{n-1}$ for all $k\geq 0$. Hence by inductive hypothesis $X_k\up{2}\in\lnta{n-1}{n-1}$. From the definition of $X_k\up{2}$ this means that, for all $s\geq 2$ and $k\geq 0$, the map
\begin{equation*}
   J_{n-2} X_{sk}\rw J_{n-2}(\pro{X_{1k}}{\di{n-2}\p{n-2}X_{0k}}{s})
\end{equation*}
is a levelwise equivalence of categories. Since this holds for every $k\geq 0$, this implies that the map
\begin{equation*}
   J_{n-1} X_s \rw J_{n-1}(\pro{X_1}{\di{n-1}\p{n-1}X_0}{s})
\end{equation*}
is a levelwise equivalence of categories. Since $\p{n}X\in\catwg{n-1}$ this implies that $X\in\lnta{n}{n}$, as claimed.

Since $X\in\catwg{n}$, by definition $X_{\seq{k}{1}{n-r}}\in \catwg{r}$, thus, from above  $X_{\seq{k}{1}{n-r}}\in \lnta{r}{r}$. Hence, by Lemma \ref{lem-lev-wg-pscat}, $X\in\lta{n}$.
\end{proof}
\begin{remark}\label{rem-lev-wg-ncat}
    If $X\in\catwg{n}$, not only the map
    \begin{equation*}
        J_{n-1}X_k\rw J_{n-1}(\pro{X_1}{\di{n-1}\p{n-1}X_0}{k})
    \end{equation*}
    is a levelwise equivalence of categories, but so is the map
    \begin{equation}\label{eq-rem-lev-wg-ncat}
        J_{n-1} X_k\rw J_{n-1}(\pro{X_1}{\di{j,n-1}\p{j,n-1}X_0}{k})
    \end{equation}
    where $\di{j,n-1}=\di{n-1}...\di{j}$ for $j=1,...,n-1$. We can see this by induction on $n$. The case $n=2$ is as in proof of Corollary \ref{cor-lev-wg-ncat}. Suppose, inductively, that the statement holds for $n-1$ and let $X\in\catwg{n}$. Then, by Proposition \ref{pro-property-WG-nfol}, $X_k\up{2}\in\catwg{n-1}$ for all $k\geq 0$. Hence by inductive hypothesis, for all $k\geq 2$ the map
    \begin{equation*}
        J_{n-2} X_{sk}\rw J_{n-2}(\pro{X_{1k}}{\di{j-1,n-2}\p{j-1,n-2}X_{0k}}{k})
    \end{equation*}
    is a levelwise equivalence of categories. Since this holds for all $k\geq 0$, this means that \eqref{eq-rem-lev-wg-ncat} is a levelwise equivalence of categories.
\end{remark}
The following lemma is used in the proof of the main result of this section, Theorem \ref{the-repl-obj-1}.
\begin{lemma}\label{lem-jn-alpha}
    Let $B \xrw{\pt_0} X \xlw{\pt_1}B$ be a diagram in $\tawg{n}$ with $X\in\cathd{n}$ and let
    \begin{equation*}
        A\xrw{\za}\tens{B}{X}\xrw{\zb}\tens{B}{X^d}
    \end{equation*}
    be maps $\tawg{n}$ (where $\zb$ is induced by the map $\zg: X\rw X^d$)  such that\mk
    \begin{itemize}
      \item [i)] $\p{n-1}\za_0$, and $\p{n-r-1}\za_{k_1...k_r\,0}$  are isomorphisms for all $1\leq r< n-1$.\mk
      \item [ii)] $(\tens{B}{X})^d_0\cong\tens{B_0^d}{X_0^d}$\mk

     \nid             $(\tens{B}{X})^d_{k_1...k_r\,0}\cong\tens{B^d_{k_1...k_r\,0}}{X^d_{k_1...k_r\,0}}$\mk
      \item [iii)] $\zb\za$ is a $n$-equivalence.\mk
    \end{itemize}
    Then $J_{n}\za$ is a levelwise equivalence of categories.
\end{lemma}
\begin{proof}
Let $x,x'\in A_0^d$. By hypothesis i) and ii),
\begin{equation*}
    A_0^d\cong(\tens{B}{X})^d_0=\tens{B_0^d}{X_0^d}\subset \tens{B_0^d}{X^d}
\end{equation*}
where the last inclusion holds because the map $\zg_0:X_0\rw (X^d)_0$ factors through $X_0^d$. Let $\za x=(a,b)\in \tens{B_0^d}{X_0^d}$ and $\za x'=(a',b')$. We have
\begin{flalign*}
&(\tens{B}{X})(x,x')=B(a,a')\tiund{X(\pt_0 a,\pt_0 a')}B(b,b')&\\
&(\tens{B}{X^d})(\zb x,\zb x')=B(a,a')\tiund{X^d(\zg a,\zg a')}B(b,b')=B(a,a')\times B(b,b')&
\end{flalign*}
where in the last equality we used the fact that $da=da'$ so that $X^d(\zg a,\zg a')=\{\cdot\}$ since $X^d$ is discrete.

The map $\zg:X\rw X^d$ factors as
\begin{equation*}
    X\rw \di{n}...\di{2}\p{2}...\p{n}X\rw \di{n}...\di{2}d p...\p{n}X=X^d
\end{equation*}
and we have
\begin{equation*}
    (\di{n}...\di{2}\p{2}...\p{n}X)(p \pt_0 a,p \pt_0 a')=X(\pt_0 a,\pt_0 a')^d\;.
\end{equation*}
Thus the map $\zb(x,x')$ factors as
\begin{equation*}
\resizebox{1.0\hsize}{!}{$
    B(a,a')\tiund{X(\pt_0 a,\pt_0 a')}B(b,b')\xrw{s} B(a,a')\tiund{X(\pt_0 a,\pt_0 a')^d}B(b,b')\xrw{t}B(a,a')\times B(b,b')\;.$}
\end{equation*}
On the other hand, since $\p{2,n}X\in\cathd{}$, the set $\p{2,n}X(p\pt_0 a,p\pt_0 a')$ contains only one element. Thus $X(\pt_0 a,\pt_0 a')^d$ is the terminal object, and $t=\Id$.
Since, by hypothesis, $\zb\za$ is a $n$-equivalence $\zb\za(x,y)$ is a $\nm$-equivalence, so by the above the composite
\begin{equation}\label{eq1-lem-jn-alpha}
    A(x,y)\xrw{\za(x,y)} B(a,a')\tiund{X(\pt a,\pt a')}B(b,b')\xrw{s} B(a,a')\tiund{X(\pt a,\pt a')^d}B(b,b')
\end{equation}
is a $\equ{n-1}$.

\smallskip
We now proceed to the rest of the proof by induction on $n$. When $n=2$, since $X(\pt_0 a,\pt_0 a')\in\cathd{}$, the map $s$ is fully faithful. Since $s\za(x,y)$ is an equivalence of categories, it is essentially surjective on objects, and therefore $s$ is essentially surjective on objects. It follows that $s$ is an equivalence of categories, and therefore such is $\za(x,y)$. Since by hypothesis $p\za_0$ is a bijection, the map $\p{2}\za$ is bijective on objects, thus $p\p{2}\za$ is surjective. From Proposition \ref{pro-n-equiv} we deduce that $\za$ is a 2-equivalence.

 Thus $\za$ satisfies the hypotheses of Proposition \ref{pro-crit-lev-nequiv} and we conclude that it is a levelwise equivalence of categories.

Suppose, inductively, that the lemma holds for $n-1$. We show that the maps \eqref{eq1-lem-jn-alpha} satisfy the inductive hypothesis.

Since $s\za(x,y)$ is a $\equ{n-1}$, inductive hypothesis iii) holds. Since by hypothesis $\p{n-1}\za_0$ is an isomorphism, so is
\begin{equation*}
    \p{n-2}\za_0(x,y)=(\p{n-1}\za_0)(x,y)
\end{equation*}
as well as
\begin{equation*}
    \p{n-r-2}\{\za_{k_1...k_r\,0}(x,y)\}=(\p{n-r-1}\za_{k_1...k_r\,0})(x,y)\;.
\end{equation*}
Thus inductive hypothesis i) holds for the maps \eqref{eq1-lem-jn-alpha}. Further, using hypothesis ii) we compute
\begin{equation*}
\begin{split}
    & ((\tens{B}{X})(x,y))^d_0=(\tens{B_1}{X_1})^d_0(x,y)= \\
    =\ & (\tens{B^d_{10}}{X^d_{10}})(x,y)=\tens{B(x,y)^d_0}{X(x,y)^d_0}
\end{split}
\end{equation*}
and similarly
\begin{equation*}
\begin{split}
    & ((\tens{B}{X})(x,y))^d_{k_1...k_r\,0}=(\tens{B}{X})^d_{1 k_1...k_r\,0}(x,y)= \\
    \cong\ & \tens{B^d_{1  {k_1...k_r\,0}}(x,y)}{X^d_{1 {k_1...k_r\,0}}(x,y)}=\\
    =\ &\tens{B_{k_1...k_r\,0}(x,y)^d}{X_{k_1...k_r\,0}(x,y)^d}\;.
\end{split}
\end{equation*}
Thus the inductive hypothesis ii) holds for the maps \eqref{eq1-lem-jn-alpha}.

We conclude by induction that $J_{n-1}\za(x,y)$ is a levelwise equivalence of categories, hence in particular it is a $\equ{n-1}$. Since by hypothesis $\p{n-1}\za_0$ is an isomorphism, so is
\begin{equation*}
    (\p{2,n}\za)_0= \p{1,n-1}\za_0
\end{equation*}
so that $\p{1,n}\za$ is surjective. Since, from above, $\za$ is a local $\nm$-equivalence, from Proposition \ref{pro-n-equiv} it follows that $\za$ is a $\nequ$. Together with hypothesis i) this shows that $\za$ satisfies the hypotheses of Proposition \ref{pro-crit-lev-nequiv} and we conclude that $J_{n}\za$ is a levelwise equivalence of categories.
\end{proof}
\begin{theorem}\label{the-repl-obj-1}
    Let $X\in \tawg{n}$, and let $r:Z\rw \di{n} \qn X$ be a map in $\tawg{n-1}$ with $Z\in\catwg{n-1}$ and consider the pullback in $\funcat{n-1}{\Cat}$
    \begin{equation*}
        \xymatrix@R=35pt @C=40pt{
        P \ar^{w}[r] \ar^{}[d] & X \ar^{\zgu{n}}[d] \\
        \dn Z \ar_{\dn r}[r] & \dn \qn X
        }
    \end{equation*}
    Then
    \begin{itemize}
      \item [a)] For all $ k\geq 2$
      \begin{equation*}
        \za: J_{n-1}P_k\rw J_{n-1}(\pro{P_1}{\di{n-1}\p{n-1}P_0}{k})
      \end{equation*}
      is a levelwise equivalence of categories.\mk

      \item [b)] $\q{n}P$ and $\p{n}P$ are in $\catwg{n-1}$.\mk

      \item [c)] $P\in\lnta{n}{n}$.\mk

      \item [d)] If $r$ is a $\nm$-equivalence then $w$ is a $\nequ$.\mk

      \item [e)] $P\in\lta{n}$.
    \end{itemize}
\end{theorem}
\begin{proof}
By induction in $n$. When $n=2$, we know by  Proposition \ref{pro-spec-plbk-pscatwg} that $P\in\tawg{2}=\lnta{2}{2}$, proving a) and c). Part b) is trivial since $\p{2}P$ and $\q{2}P$ are in $\Cat$. Part a) follows from Lemma \ref{pro-spec-plbk-pscatwg} c). Part e) holds since $tawg{2}=\lta{2}$

\medskip
Suppose, inductively that the theorem holds for $n-1$.

\bk

a) We first show that for each $k\geq 2$
\begin{equation*}
    \pro{P_1}{\di{n-1}\p{n-1}P_0}{k}\in \tawg{n-1}\;.
\end{equation*}
We illustrate this for $k=2$, the case $k>2$ being similar. Since $p$ commutes with pullbacks over discrete objects we have
\begin{equation*}
    \p{n}P=Z\tiund{\q{n}X}\p{n}X
\end{equation*}
and, (since $\q{n-1}X_0=\p{n-1}X_0$ as $X_0\in\cathd{n-1}$)
\begin{equation*}
    \p{n-1}P_0=Z_0\tiund{\q{n-1}X_0}\p{n-1}X_0=Z_0\;.
\end{equation*}
Also, $P_1=\di{n-1}Z_1\tiund{\di{n-1}\q{n-1}X_1}X_1$. Therefore
\begin{equation*}
\begin{split}
    & \tens{P_1}{\di{n-1}\p{n-1}P_0}= \\
    =\ & \tens{(\di{n-1}Z_1\tiund{\di{n-1}\q{n-1}X_1}X_1)}{\di{n-1}Z_0}=\\
    =\ & \di{n-1}(\tens{Z_1}{Z_0})\tiund{\di{n-1}\q{n-1}({X_1}\times{X_1})} ({X_1}\times{X_1})\;.
\end{split}
\end{equation*}
 By Proposition \ref{pro-spec-plbk-pscatwg}, this is an object of $\tawg{n-1}$.

 The induced Segal map $\hmu{2}$ for $P$ can therefore be written as composite of maps in $\tawg{n-1}$
\begin{equation}\label{eq-the-repl-obj-1}
    P_2 \xrw{\za} \tens{P_1}{\di{n-1}\p{n-1}P_0} \xrw{\zb} \tens{P_1}{P_0^d}\;.
\end{equation}
We show that the maps \eqref{eq-the-repl-obj-1}  satisfies the hypotheses of Lemma \ref{lem-jn-alpha}.

Since $P\in\tawg{n}$, $\hmu{2}=\zb\za$ is a $\equ{n-1}$, so hypothesis iii) of Lemma \ref{lem-jn-alpha} holds for the maps \eqref{eq-the-repl-obj-1} .

To check hypothesis i), note that
\begin{equation*}
\begin{split}
    & \p{n-2}P_{20}=\p{n-2}(\di{n-2}Z_{20}\tiund{\di{n-2}\q{n-2}X_{10}} \di{n-2}\p{n-2}X_{10})=Z_{20} \\
    & \p{n-2}(\tens{P_1}{\di{n-1}\p{n-1}P_0})_0=\\
    =\ &\tens{\p{n-2}P_{10}}{\p{n-2}P_{00}}=\tens{Z_{10}}{Z_{00}}\cong Z_{20}
\end{split}
\end{equation*}
where the last isomorphism holds since $Z\in\catwg{n-1}$.
Hence $\p{n-2}\za_0$ is an isomorphism. Similarly
\begin{equation*}
    P_{s_1...s_r}=Z_{s_1...s_r}\tiund{\q{n-r-1}X_{s_1...s_r}}X_{s_1...s_r}
\end{equation*}
\begin{equation*}
    \p{n-r-1}P_{s_1...s_r}=Z_{s_1...s_r}\tiund{\q{n-r-1}X_{s_1...s_r}}X_{s_1...s_r}\;.
\end{equation*}
Thus
\begin{equation*}
\begin{split}
    & \p{n-r-2}P_{2 k_1...k_r 0}=Z_{2 k_1...k_r 0} \\
    & \p{n-r-2}(\tens{P_1}{\di{n-1}\p{n-1}P_0})_{k_1...k_r 0}=\\
    =\ & \p{n-r-2}(\tens{P_{1 k_1...k_r 0}}{\di{n-r-1}\p{n-r-1}P_{0 k_1...k_r 0}})=\\
    =\ & \tens{\p{n-r-2}P_{1 k_1...k_r 0}}{\p{n-r-1}P_{0 k_1...k_r 0}}=\\
    =\ & \tens{Z_{1 k_1...k_r 0}}{Z_{0 k_1...k_r 0}}\cong Z_{2 k_1...k_r 0}\;.
\end{split}
\end{equation*}
where the last isomorphism holds since $Z\in\catwg{n-1}$. This shows that $\p{n-r-2}\za_{k_1...k_r 0}$ is an isomorphism, proving hypothesis i) of Lemma \ref{lem-jn-alpha} for the maps \eqref{eq-the-repl-obj-1}.

To check hypothesis ii) of Lemma \ref{lem-jn-alpha} note that $P^d_{20}=Z^d_{20}$ while
\begin{equation*}
\begin{split}
    & (\tens{P_1}{\di{n-1}\p{n-1}P_0})^d_0=(\tens{P_{10}}{\di{n-2}\p{n-2}P_{00}})^d_0= \\
    =\ & (\tens{Z_{10}}{\di{n-2}\p{n-2}Z_{00}})^d=\tens{Z_{10}^d}{Z_{00}^d}\cong Z_{20}^d
\end{split}
\end{equation*}
where the last isomorphism holds because $Z\in\catwg{n}$ (apply Remark \ref{rem-eq-def-wg-ncat} to $Z_0\up{2}$ which is an object of $\catwg{n-2}$ by Proposition \ref{pro-property-WG-nfol}). Similarly
\begin{equation*}
\begin{split}
    & (P_2)^d_{k_1...k_r 0}=Z^d_{2 k_1...k_r 0} \\
    & (\tens{P_1}{\di{n-1}\p{n-1}P_0})^d_{k_1...k_r 0}=\\
    =\ & (\tens{Z_{1 k_1...k_r 0}}{\di{n-1}\p{n-1}Z_{0 k_1...k_r 0}})^d=\\
    =\ & \tens{Z^d_{1 k_1...k_r 0}}{Z^d_{0 k_1...k_r 0}}=Z^d_{2 k_1...k_r 0}\;.
\end{split}
\end{equation*}
where the last equality holds because $Z\in\catwg{n-1}$ and therefore $Z_{k_0...k_r}\in\catwg{n-r-1}$ by applying Remark \ref{rem-eq-def-wg-ncat} to $(Z_{k_0...k_r})_0\up{2}$ which is an object of $\catwg{n-r-1}$ by Proposition \ref{pro-property-WG-nfol}. This proves that hypothesis ii) of Lemma \ref{lem-jn-alpha} holds for \eqref{eq-the-repl-obj-1}.

So all hypotheses of Lemma \ref{lem-jn-alpha} holds for the maps \eqref{eq-the-repl-obj-1}  and we conclude that $J_{n-1}\za$ is a levelwise equivalence of categories.

The proof that, for each $s > 2$ the map
\begin{equation*}
    P_s\rw \pro{P_1}{\di{n-1}\p{n-1}P_0}{s}
\end{equation*}
is a levelwise equivalence of categories is similar.

\bk
b) We have $\q{n}P=Z\in\catwg{n}$, $\p{n}P\in\tawg{n-1}$ since $P\in \tawg{n}$ by Proposition \ref{pro-spec-plbk-pscatwg}. We now show that $\p{n}P$ satisfies the hypotheses of Lemma \ref{lem-x-in-tawg-x-in-catwg}, which then shows that $\p{n}P\in\catwg{n-1}$.

We have the pullback in $\funcat{n-2}{\Cat}$ for each $k\geq 0$,
\begin{equation*}
\xymatrix{
P_k \ar[rr] \ar[d] && X_k \ar[d]\\
Z_k \ar[rr] && \di{n-1}\p{n-1}X_k
}
\end{equation*}
which\, satisfies\, the\, induction\, hypothesis. \,Therefore, \,by \,induction,\, $\p{n-1}P_k\in\catwg{n-2}$ which is hypothesis a) of Lemma \ref{lem-x-in-tawg-x-in-catwg}. Since, by part a), the map
\begin{equation*}
    t_2:P_2\rw \tens{P_1}{\di{n-1}\p{n-1}P_0}
\end{equation*}
is a levelwise equivalence of categories, it induces an isomorphism
\begin{equation*}
    \p{n-1}P_2=\p{n-1}(\tens{P_1}{\di{n-1}\p{n-1}P_0})=\tens{\p{n-1}P_1}{\p{n-1}P_0}
\end{equation*}
and similarly all the other Segal maps of $\p{n}P$ are isomorphisms. This proves hypothesis b) of Lemma \ref{lem-x-in-tawg-x-in-catwg} for $\p{n}X$.

To prove hypothesis c) of Lemma \ref{lem-x-in-tawg-x-in-catwg}, we first note that
\begin{equation}\label{eq0-the-repl-obj-1}
    \tens{P_1}{\di{j,n-1}\p{j,n-1}P_0}\in \tawg{n-1}
\end{equation}
where we denote
\begin{equation*}
    \p{j,n-1}=\p{j}...\p{n-1}, \qquad \di{j,n-1}=\di{n-1}...\di{j}\;.
\end{equation*}
In fact, since $\p{n-1}P_0=Z_0$ we have
\begin{equation}\label{eq1-the-repl-obj-1}
\begin{split}
    & \tens{P_1}{\di{j,n-1}\p{j,n-1}P_0}= \\
    & \resizebox{1.0\hsize}{!}{$ \;= \tens{(\di{n-1}Z_1 \tiund{\di{n-1}\q{n-1}X_1} X_1)}{\di{j,n-1}\p{j,n-1}Z_0}=$}\\
   & = \{\tens{\di{n-1}Z_1}{\di{j,n-1}\p{j,n-1}Z_0}\}\tiund{\di{n-1}\q{n-1}(X_1\times X_1)} (X_1\times X_1)=\\
   & = \di{n-1}(\tens{Z_1}{\di{j,n-2}\p{j,n-1}Z_0})\tiund{\di{n-1}\q{n-1}(X_1\times X_1)} (X_1\times X_1)\;.
\end{split}
\end{equation}
By Remark \ref{rem-lev-wg-ncat}, $\tens{Z_1}{\di{j,n-2}\p{j,n-2}Z_0}\in\catwg{n-1}$; by Proposition \ref{pro-spec-plbk-pscatwg} and  \eqref{eq1-the-repl-obj-1} we conclude that \eqref{eq0-the-repl-obj-1} holds. It follows that
\begin{equation*}
\begin{split}
    & \p{j+1,n-1}(\tens{P_1}{\di{j,n-1}\p{j,n-1}P_0})= \\
    =\ & \tens{\p{j+1,n-1}P_1}{\di{j}\p{j}\p{j+1,n-1}P_0}\in\tawg{j}\;.
\end{split}
\end{equation*}
Since $P\in\tawg{n}$, $\p{j+1,n}P\in\tawg{j}$ so that the induced Segal map
\begin{equation*}
    \p{j+1,n-1}P_2\rw \tens{\p{j+1,n-1}P_1}{(\p{j+1,n-1}P_0)^d}
\end{equation*}
are $\equ{j-1}$s. From above, this map factorizes as composite of maps in $\tawg{j}$
\begin{equation*}
\begin{split}
    & \p{j+1,n-1}P_2 \xrw{\;\za\;} \tens{\p{j+1,n-1}P_1}{\di{j}\p{j}\p{j+1,n-1}P_0} \xrw{\;\zb\;}\\
   \xrw{\;\zb\;}\ &\tens{\p{j+1,n-1}P_1}{(\p{j+1,n-1}P_0)^d}\;.
\end{split}
\end{equation*}
It is not difficult to check that this satisfies the hypotheses of Lemma  \ref{lem-jn-alpha}; thus we conclude that $J_j\za$ is a levelwise equivalence of categories. Therefore $\p{j}\za$ is an isomorphism, that is
\begin{equation*}
\begin{split}
    & \p{j,n-1}P_2 \cong \p{j}(\tens{\p{j+1,n-1}P_1}{\di{j}\p{j}\p{j+1,n-1}P_0})= \\
    =\ & \tens{\p{j,n-1}P_1}{\p{j,n-1}P_0}\;.
\end{split}
\end{equation*}
Similarly one shows that all the other Segal maps for $\p{j+1,n}P$ are isomorphisms, which proves condition c) in Lemma \ref{lem-x-in-tawg-x-in-catwg} for $\p{n}P$. Thus by Lemma \ref{lem-x-in-tawg-x-in-catwg} we conclude that $\p{n}P\in\catwg{n}$, proving b).

\bk
c) By definition of $\lnta{n}{n}$, this follows from a) and b).

\bk

d) Consider the commuting diagram in $\tawg{n}$
\begin{equation*}
\xymatrix@R=35pt @C=40pt{
\di{n}Z \ar^{\di{n}r}[r] \ar_{\di{n}r}[d] & \di{n}\q{n}X \ar@{=}^{}[d] & X \ar^{}[l] \ar@{=}^{}[d] \\
\di{n}\q{n}X \ar@{=}^{}[r] & \di{n}\q{n}X & X  \ar^{}[l]
}
\end{equation*}
By hypothesis, $\di{n}r$ is a $\nequ$. Thus by Proposition \ref{pro-spec-plbk-pscatwg} the induced map of pullbacks
\begin{equation*}
    P=\di{n}Z \tiund{\di{n}\q{n}X} X \xrw{w} \di{n}\q{n}X \tiund{\di{n}\q{n}X} X =X
\end{equation*}
is a $\nequ$, as required.

\bk

e) By c) and by Lemma \ref{lem-lev-wg-pscat}, to show that $P\in\lta{n}$ it is enough to show that, for each $\seqc{k}{1}{n-s}\in\Dop$, $P_{\seqc{k}{1}{n-s}}\in\lnta{s}{s}$, $1< s \leq n-1$. We have a pullback in $\funcat{n-1} {\Cat}$
\begin{equation}\label{eq2-the-repl-obj-1}
\xymatrix@R=35pt @C=45pt{
P_{\seq{k}{1}{n-s}} \ar^{w_{\seq{k}{1}{n-s}}}[r] \ar^{}[d] & X_{\seq{k}{1}{n-s}} \ar^{}[d] \\
 \di{s} Z_{\seq{k}{1}{n-s}} \ar_{\di{s}r_{\seq{k}{1}{n-s}}}[r] & \di{s}\q{s} X_{\seq{k}{1}{n-s}}
 }
\end{equation}
where $X_{\seq{k}{1}{n-s}}\in \tawg{s}$ (since $X\in\tawg{n}$) and $Z_{\seq{k}{1}{n-s}}\in \catwg{s-1}$ (since $Z\in\catwg{n-1}$).

Thus \eqref{eq2-the-repl-obj-1} satisfies the hypotheses of the theorem and we conclude from c) that $P_{\seq{k}{1}{n-s}}\in\lnta{s}{s}$, as required.

\end{proof}
\medskip

\section{Ridigifying weakly globular Tamsamani ${n}$-categories}\label{sec-wg-tam-to-psefun}
In this section we show the main result of the paper, Theorem \ref{the-funct-Qn}, establishing the existence of a rigidification functor
\begin{equation*}
  Q_n:\tawg{n}\rw \catwg{n}
\end{equation*}
replacing $X\in\tawg{n}$ with an $n$-equivalent object $Q_n X$.

 The construction of the functor $Q_n$ uses three main ingredients: the first is the approximation up to $n$-equivalence of any object of $\tawg{n}$ with an object of $\lta{n}$ established in the previous section. The second ingredient is a functor $Tr_n$  from the category $\lta{n}$ to the category of Segalic pseudo-functors, which we establish in this section in Theorem \ref{the-XXXX}. The third ingredient is the functor from Segalic pseudo-functors to weakly globular $n$-fold categories from \cite{Pa2} recalled in Theorem \ref{the-strict-funct}.

 The construction of the functor $Tr_n$ is based on a general categorical technique  which we recalled in Lemma \ref{lem-PP}. We can apply this technique to objects of the category $\lta{n}$ thanks to a property of the latter, which we establish in Lemma \ref{lem-maps-nu-eqcat}.

 \bigskip

 We first fix a notation for the induced Segal maps of categories for an object $X$ of $\tawg{n}$ when all but one simplicial directions in $J_nX$ are fixed.
\begin{notation}\label{not-ind-seg-map}
    Let $X\in\tawg{n}$, $\uk=(k_1,\ldots,k_{n-1})\in\dop{n-1}$, $1\leq i\leq n-1$. Then there is $X^i_{\uk}\in\funcat{}{\Cat}$ with
    \begin{equation*}
        (X^i_{\uk})_r=X_{\uk(r,i)}=X_{k_1...k_{i-1}r k_{i+1}...k_{n-1}}
    \end{equation*}
    so that $(X^i_{\uk})_{k_i}=X_{\uk}$. Since $X_{k_1,\ldots,k_{i-1}}\in\tawg{n-i+1}$, $X_{k_1,\ldots,k_{i-1}0}\in\cathd{n-i}$ and thus by \cite[Lemma 3.1]{Pa1}
    \begin{equation*}
        X_{k_1...k_{i-1} 0 k_{i+1}...k_{n-1}}= X_{\uk(0,i)}\in\cathd\;.
    \end{equation*}
    We therefore obtain induced Segal maps in $\Cat$ for all $k_i\geq 2$.
    \begin{equation}\label{eq1-not-ind-seg-map}
        \nu(\uk,i):X_{\uk}\rw \pro{X_{\uk(1,i)}}{X^d_{\uk(0,i)}}{k_i}\;.
    \end{equation}
\end{notation}
\begin{lemma}\label{lem-maps-nu-eqcat}
    Let $X\in\lta{n}$; then for each $\uk\in\dop{n-1}$, $1\leq i\leq n-1$ and $k_i\geq 2$ the maps $ \nu(\uk,i)$ as in \eqref{eq1-not-ind-seg-map} are equivalences of categories.
\end{lemma}
\begin{proof}
By induction on $n$. It holds when $n=2$ since $\lta{2}=\tawg{2}$ and, for $k\geq 2$, $\nu(k,1)$ is the induced Segal map
\begin{equation*}
    X_k\rw \pro{X_1}{X_0^d}{k}
\end{equation*}
which is an equivalence of categories by definition of $\tawg{2}$.

Suppose the lemma holds for $(n-1)$ and let $X$ be as in the hypothesis. Consider first the case $1<i\leq n-1$. Since $X\in\lta{n}$, by definition $X_{k_1}\in\lta{n-1}$. Denoting $\ur=(k_2...k_{n-1})$ we have
\begin{equation}\label{eq2-not-ind-seg-map}
\begin{split}
    & X_{\uk}=(X_{k_1})_{\ur} \\
    & X_{\uk(1,i)}=(X_{k_1})_{\ur(1,i-1)}\qquad X_{\uk(0,i)}=(X_{k_1})_{\ur(0,i-1)}\;.
\end{split}
\end{equation}
The induction hypothesis applied to $Y=X_{k_1}$ implies the equivalence of categories for each $r_{i-1}\geq 2$
\begin{equation*}
    Y_{\ur}\simeq \pro{Y_{\ur(1,i-1)}}{Y^d_{\ur(0,i-1)}}{r_{i-1}}\;.
\end{equation*}
By \eqref{eq2-not-ind-seg-map} this means that $\nu(\uk,i)$ is an equivalence of categories for all $1<i\leq n-1$.

\nid Consider the case $i=1$. By definition of $\lta{n}$ the map in $\funcat{n-2}{\Cat}$
\begin{equation*}
    X_{k_1}\rw\pro{X_1}{\di{n-1}\p{n-1}X_0}{k_1}
\end{equation*}
is a levelwise equivalence of categories for each $k_i\geq 2$. Therefore, for each $\ur=(k_2...k_{n-1})$ there is an equivalence of categories.
\begin{equation*}
    (X_{k_1})_{\ur}\simeq \pro{(X_1)_{\ur}}{d p (X_0)_{\ur}}{k_1}\;.
\end{equation*}
Since $(X_1)_{\ur}=X_{\uk(1,1)}$ and $dp(X_0)_{\ur}=X^d_{\uk(0,1)}$ we obtain the equivalence of categories
\begin{equation*}
    X_{\uk}\simeq \pro{X_{\uk(1,1)}}{X^d_{\uk(0,1)}}{k_i}\;.
\end{equation*}
In conclusion $\nu(\uk,i)$ is an equivalence of categories for all $1\leq i\leq n-1$.
\end{proof}
\begin{theorem}\label{the-XXXX}
    There is a functor
    \begin{equation*}
        Tr_{n}: \lta{n} \rw \segpsc{n-1}{\Cat}
    \end{equation*}
    together with a pseudo-natural transformation $t_n (X):Tr_{n}X\rw X$ for each $X\in\lta{n}$ which is a levelwise equivalence of categories.
\end{theorem}
\begin{proof}
Given $X\in\lta{n}$ define the diagram $Tr_{n} X$ in $[ob(\dop{n-1}),\Cat]$ as follows: for each $\uk\in\dop{n-1}$ and $1\leq i\leq n-1$
\begin{itemize}
  \item [i)] If $k_j=0$ for some $1\leq j\leq n-1$
  \begin{equation*}
    (Tr_{n} X)_{\uk}=X^d_{\uk}\;.
  \end{equation*}

  \item [ii)] If $k_j\neq 0$ for all $1\leq j\leq n-1$ and $k_i=1$,
  \begin{equation*}
    (Tr_{n} X)_{\uk}=X_{\uk(1,i)}=X_{\uk}\;.
  \end{equation*}

  \item [iii)] If $k_j\neq 0$ for all $1\leq j\leq n-1$ and $k_i>1$,
  \begin{equation*}
    (Tr_{n} X)_{\uk}=\pro{X_{\uk(1,i)}}{X^d_{\uk(0,i)}}{k_i} \;.
  \end{equation*}

  \end{itemize}
  We claim that, for all $\uk\in \dop{n-1}$ there is an equivalence of categories
  \begin{equation*}
    (Tr_{n} X)_{\uk}=X_{\uk}\;.
  \end{equation*}
  In fact, in case i), $X_{\uk}\in\cathd{}$ thus there is an equivalence of categories
  \begin{equation*}
    (Tr_{n} X)_{\uk}= X^d_{\uk}\simeq X_{\uk}\;.
  \end{equation*}
  In case ii) we have
  \begin{equation*}
    (Tr_{n} X)_{\uk}=X_{\uk(1,i)}=X_{\uk}\;.
  \end{equation*}
  In case iii) by Lemma \ref{lem-maps-nu-eqcat}  we have an equivalence of categories
  \begin{equation*}
    (Tr_{n} X)_{\uk}=\pro{X_{\uk(1,i)}}{X^d_{\uk(0,i)}}{k_i}\simeq X_{\uk} \;.
  \end{equation*}
  This proves the claim.

We can therefore apply Lemma \ref{lem-PP} with $\clC=\dop{n-1}$ and conclude that $Tr_{n}X$ lifts to a pseudo-functor
\begin{equation*}
    Tr_{n}X \in \psc{n-1}{\Cat}
\end{equation*}
with $(Tr_{n} X)_{\uk}$ as in i), ii) and iii). We check that $Tr_{n}X$ is a Segalic pseudo-functor.

Conditions a) in Definition \ref{def-seg-ps-fun} is satisfied by construction. As for condition b), let $\uk\in\dop{n-1}$ be such that $k_j=0$ for some $1\leq j\leq n-1$ and let $k_i\geq 2$ for $1\leq i\leq n$. Since by Lemma \ref{lem-maps-nu-eqcat} there is an equivalence of categories
\begin{equation*}
  X_{\uk}\simeq \pro{X_{\uk(1,i)}}{X^d_{\uk(0,i)}}{k_i}
\end{equation*}
and since $X_{\uk}\in \cathd{}$, $X_{\uk(1,i)}\in \cathd{}$, there is an isomorphism
\begin{equation*}
    X^d_{\uk}\cong dp( \pro{X_{\uk(1,i)}}{X^d_{\uk(0,i)}}{k_i})\cong \pro{X^d_{\uk(1,i)}}{X^d_{\uk(0,i)}}{k_i}
\end{equation*}
which is the same as
\begin{equation*}
    (Tr_{n}X )_{\uk}\cong\pro{(Tr_{n}X )_{\uk(1,i)}}{(Tr_{n}X )_{\uk(0,i)}}{k_i}\;.
\end{equation*}
Thus condition b) in Definition \ref{def-seg-ps-fun} holds in this case.

Suppose $k_j\neq 0$ for all $1\leq j\leq n-1$, $k_i\geq 2$. Then
\begin{equation*}
\begin{split}
    & (Tr_{n}X )_{\uk}=\pro{X _{\uk(1,i)}}{X^d_{\uk(0,i)}}{k_i}\cong \\
 \cong \,  & \pro{(Tr_{n}X )_{\uk(1,i)}}{(Tr_{n}X )_{\uk(0,i)}}{k_i} \;.
\end{split}
\end{equation*}
So condition b) in Definition \ref{def-seg-ps-fun} also holds in this case. As for condition c), we note that the equivalence of categories $(Tr_{n} X)_{\uk}\simeq X_{\uk}$ for each $\uk\in\dop{n-1}$ implies the isomorphism
\begin{equation*}
    p (Tr_{n} X)_{\uk} \cong p X_{\uk} =(\p{n}X)_{\uk}\;.
\end{equation*}
Since $X\in\lta{n}$, $\p{n}X\in\catwg{n-1}$ hence $\p{n}Tr_{n} X\in\catwg{n-1}$. By definition this means that $Tr_{n} X\in\segpsc{n-1}{\Cat}$. By Lemma \ref{lem-PP} there is a morphism in $\psc{n-1}{\Cat}$, \, $t_n(X):Tr_{n} X\rw X$ which is a levelwise equivalence of categories.
\end{proof}
We now prove the main result of this paper, which is the existence of a rigidification functor from $\tawg{n}$ to $\catwg{n}$ replacing any object of $\tawg{n}$ with a suitably equivalent object in $\catwg{n}$.
\begin{theorem}\label{the-funct-Qn}
    There is a functor
    \begin{equation*}
        Q_n:\tawg{n} \rw \catwg{n}
    \end{equation*}
    and for each $X\in\tawg{n}$ a morphism in $\tawg{n}$  $s_n(X):Q_n X\rw X$, natural in $X$, such that $(s_n(X))_{k}$ is a $(n-1)$-equivalence for all $k\geq0$. In particular, $s_{n}(X)$ is an $n$-equivalence.
\end{theorem}
\begin{proof}
By induction on $n$. When $n=2$, let $Q_2$ be the composite
\begin{equation*}
    Q_2: \tawg{2}\xrw{Tr_{2}} \segpsc{}{\Cat} \xrw{\St} \catwg{2}
\end{equation*}
where $Tr_{2}$ is as in Theorem \ref{the-XXXX}. By Theorem \ref{the-strict-funct}, $Q_2X \in \catwg{2}$. Recall \cite{Lack} that the strictification functor
\begin{equation*}
    \St:\psc{}{\Cat} \rw \funcat{}{\Cat}
\end{equation*}
is left adjoint to the inclusion
\begin{equation*}
    J: \funcat{}{\Cat} \rw \psc{}{\Cat}
\end{equation*}
and that the components of the unit are equivalences in $\psc{}{\Cat}$. By Theorem \ref{the-XXXX} there is a morphism in $\psc{}{\Cat}$
\begin{equation*}
    t_2(X):Tr_{2}X \rw JX \;.
\end{equation*}
By adjunction this corresponds to a morphism in $\funcat{}{\Cat}$
\begin{equation*}
    Q_2 X=\St Tr_{2} X\xrw{s_2(X)} X
\end{equation*}
making the following diagram commute
\begin{equation*}
    \xymatrix@C=50pt@R=50pt{
    Tr_{2}X \ar^(0.4){\eta}[r] \ar_{t_2(X)}[dr] & J\St Tr_{2}X
    \ar^{J s_2(X)}[d] \\
    &  JX
    }
\end{equation*}
Since $\eta$ and $t_2(X)$ are levelwise equivalences of categories, such is $s_2(X)$.

Suppose, inductively, that we defined $Q_{n-1}$. Define the functor
\begin{equation*}
    P_n:\tawg{n}\rw\lta{n}
\end{equation*}
as follows. Given $X\in\tawg{n}$, consider  the pullback in $\funcat{n-1}{\Cat}$
\begin{equation*}
    \xymatrix@C=70pt@R=30pt{
    P_n X \ar^{w(X)}[r] \ar[d] & X \ar^{\zg_{n}}[d] \\
    \di{n}Q_{n-1}\q{n}X \ar_{\di{n}s_{n-1}(\q{n}X)}[r] & \di{n}\q{n}X
    }
\end{equation*}
By Theorem \ref{the-repl-obj-1}, $P_n X\in\lta{n}$. Define
\begin{equation*}
Q_n X=\St Tr_{n} P_n X.
\end{equation*}
By Theorem \ref{the-strict-funct}, $Q_nX \in \catwg{n}$. Let $s_n(X):Q_n X\rw X$ be the composite
\begin{equation*}
s_n(X): Q_n X \xrw{h_n(P_n X)} P_n X \xrw{w(X)} X\;.
\end{equation*}
 where  the morphism in $\funcat{n-1}{\Cat}$
\begin{equation*}
    Q_n X= \St Tr_{n}P_n X \xrw{h_n(P_n X)} P_n X
\end{equation*}
 corrersponds by adjunction to the morphism in $\psc{n-1}{\Cat}$
\begin{equation*}
    Tr_{n}P_n X \xrw{t_n(P_n X)} J P_n X
\end{equation*}
(where $t_n(P_n X)$ is as in Theorem \ref{the-XXXX}) such that the following diagram commutes
\begin{equation*}
    \xymatrix@C=50pt@R=50pt{
    Tr_{n}P_n X \ar^(0.35){\eta}[r] \ar_{t_n(P_n X)}[dr] & J \St Tr_{n}P_n
    X= J Q_n X \ar^{J h_n(P_n X)}[d]\\
    & J P_n X
    }
\end{equation*}
We need to show that $(s_n(X))_{k}$ is an $(n-1)$-equivalence. Since $\eta$ and $t_n(P_n X)$ are levelwise equivalence of categories, such is $h_n(P_n X)$ so in particular $(h_n(P_n X))_{k}$ is a levelwise equivalence of categories, and thus is a $(n-1)$-equivalence (see Remark \ref{rem-local-equiv}).

Since pullbacks in $\funcat{n-1}{\Cat}$ are computed pointwise, there is a pullback in $\funcat{n-2}{\Cat}$
\begin{equation*}
\xymatrix@C=90pt@R=40pt{
(P_n X)_k \ar^{(w(X))_k}[r] \ar[d] & X_k \ar[d]\\
\di{n-1}(Q_{n-1}\q{n}X)_k \ar_(0.55){\di{n-1}(s_{n-1}(\q{n}X))_k}[r] & \di{n-1}\q{n-1}X_k
}
\end{equation*}
where $X_k\in\tawg{n-1}$ (since $X\in\tawg{n}$) and $(Q_{n-1}\q{n}X)_k\in\catwg{n-2}$ (since $Q_{n-1}\q{n}X\in\catwg{n-1}$) and, by induction hypothesis, $(s_{n-1}(\q{n}X))_k$ is a $(n-2)$-equivalence. It follows by Theorem \ref{the-repl-obj-1} that $(w(X))_k$ is a $(n-1)$-equivalence.

In conclusion both $(h_n(P_n X))_k$ and $(w(X))_k$ are $(n-1)$-equivalences so by Proposition \ref{pro-n-equiv} such is their composite
\begin{equation*}
    (s_n(X))_k : (Q_n X)_k \xrw{(h_n(P_n X))_k} (P_n X)_k \xrw{(w(X))_k} X_k
\end{equation*}
as required. By Lemma \ref{lem-flevel-fneq}, it follows that $s_n(X)$ is an $n$-equivalence.

\end{proof}
\begin{corollary}\label{cor-the-funct-Qn}
    The functor $Q_n:\tawg{n}\rw \catwg{n}$ induces an equivalence of categories
    \begin{equation}\label{eq-cor-the-funct-Qn}
        \tawg{n}\bsim^n\; \simeq\; \catwg{n}\bsim^n
    \end{equation}
    after localization with respect to the $n$-equivalences.
\end{corollary}
\begin{proof}
Let $i:\catwg{n}\hookrightarrow \tawg{n}$ denote the embedding of $\catwg{n}$ into $\tawg{n}$. Given $X\in\tawg{n}$, by Theorem \ref{the-funct-Qn} there is an $n$-equivalence in $\tawg{n}$ $i Q_n X\rw X$, therefore $i Q_n X\cong X$ in $\tawg{n}\bsim^n$.

Let $Y\in\catwg{n}$; then $i Y\in\tawg{n}$ so by Theorem \ref{the-funct-Qn} there is an $n$-equivalence in $\tawg{n}$ $i Q_n i Y \rw i Y$. Since $i$ is fully faithful, $ Q_n i Y \rw Y$ is an $n$-equivalence in $\catwg{n}$. It follows that $Q_n i Y \cong Y$ in $\catwg{n}\bsim^n$.

In conclusion $Q_n$ and $i$ induce the equivalence of categories \eqref{eq-cor-the-funct-Qn}.
\end{proof}


\end{document}